\documentclass[reqno]{amsart}

\usepackage{amsmath,amsthm,amsfonts, amssymb,amscd}
\usepackage[all]{xy}
\numberwithin{equation}{section}

\newtheorem{theorem}{Theorem}[section]
\newtheorem{example}[theorem]{Example}
\newtheorem{remark}[theorem]{Remark}
\newtheorem{corollary}[theorem]{Corollary}
\newtheorem{proposition}[theorem]{Proposition}

\newcommand{\dx}{\, \mbox{\rm d}}
\newcommand{\Ker}{\mbox{\rm Ker}}

\newcommand{\Rad}{\mbox{\rm Rad}}

\newcommand{\Infinit}{\mbox{\rm Infinit}}

\newcommand{\Con}{\mbox{\rm Con}}

\begin{document}
\title[States on EMV-algebras]{States on EMV-algebras}
\author[Anatolij Dvure\v{c}enskij, Omid Zahiri]{Anatolij Dvure\v{c}enskij$^{1,2}$, Omid Zahiri$^3$}

\date{}%
\thanks{The first author gratefully acknowledges the support by the grant of
the Slovak Research and Development Agency under contract APVV-16-0073, by the grant VEGA No. 2/0069/16 SAV as well as by GA\v{C}R 15-15286S}
\address{$^1$Mathematical Institute, Slovak Academy of Sciences, \v{S}tef\'anikova 49, SK-814 73 Bratislava, Slovakia}
\address{$^2$Palack\' y University Olomouc, Faculty of Sciences, t\v r. 17.listopadu 12, CZ-771 46 Olomouc, Czech Republic}
\address{$^3$University of Applied Science and Technology, Tehran, Iran}
\email{\tt dvurecen@mat.savba.sk, zahiri@protonmail.com}


\maketitle

\begin{abstract}
We define a state as a $[0,1]$-valued, finitely additive function attaining the value $1$ on an EMV-algebra, which is an algebraic structure close to MV-algebras, where the top element is not assumed. We show that states always exist, the extremal states are exactly state-morphisms. Nevertheless the state space is a convex space that is not necessarily compact, a variant of the Krein--Mil'man theorem saying states are generated by extremal states, is proved. We define a weaker form of states, pre-states and strong pre-states, and also Jordan signed measures which form a Dedekind complete $\ell$-group. Finally, we show that every state can be represented by a unique regular probability measure, and a variant of the Horn--Tarski theorem is proved.
\end{abstract}

\section{Introduction}

Probability reasoning on MV-algebras has been started in \cite{Mun2} by states, where a state is a finitely additive and positive function $s$ on an MV-algebra $M$ that is normalized, i.e. $s(1)=1$, and $s(x\oplus y)=s(x)+s(y)$ whenever $x\odot y = 0$ for $x,y \in M$. It means averaging the truth-value in \L ukasiewicz logic. This is a special case of states on effect algebras because every MV-algebra can be studied also as an effect algebra. For more info about states on effect algebras see \cite{DvPu}, and about states on a non-commutative form of MV-algebras see \cite{DvuS}. We note that a state for effect algebras is a crucial notion because effect algebras introduced in \cite{FoBe} have been used for modeling uncertainties in quantum mechanical measurements.

There is also another approach to probability reasoning of MV-algebras.
In \cite{FlMo}, the authors find an algebraizable logic whose equivalent algebraic semantics is the variety of state MV-algebras. In other words, they expanded MV-algebras by a unary operator, called an internal state or a state-operator, whose properties resemble the properties of a state.

The probability methods used in MV-algebras have been expanded in the last 10--15 years.

We note that a probability measure on a measurable space $(\Omega,\mathcal S)$, where $\Omega$ is a non-void set and $\mathcal S$ is a $\sigma$-algebra of subsets of $\Omega$, is by \cite{Kol} a $\sigma$-additive probability measure $P$, however due to de Finetti, a probability measure has to be only a finitely additive measure. But due to \cite{Kro,Pan}, these approaches are more-less equivalent because every finitely additive measure even on an MV-algebra can be represented by a unique regular $\sigma$-additive probability measure. Such a result was extended also for effect algebras, see \cite{241,247}. An analogous result will be established in the present paper also for EMV-algebras.

In \cite{DvZa}, the authors introduced EMV-algebras (extended MV-algebras) which locally resemble MV-algebras, but no top element is guaranteed. They extend generalized Boolean algebras, or equivalently, Boolean rings. We extended \L ukasiewicz type algebraic structures with incomplete total information which is complete only locally: Conjunction and disjunctions exist but negation exists only in a local sense, i.e. negation of $a$ in $b$ exists whenever $a\le b$, but the total negation of the event $a$ is not assumed.

The basic representation theorem for EMV-algebras, see \cite[Thm 5.21]{DvZa}, says an EMV-algebra is either an MV-algebra or we can find an MV-algebra $N$ where the original EMV-algebra can be embedded as a maximal ideal of the MV-algebra $N$. This result is crucial for our reasoning.

The main aim of the paper is to introduce and study states for EMV-algebras even if they do not possess a top element. Then if $s$ is a state on an EMV-algebra and $a$ is an event, then $s(a)$ will represent averaging the truth-value of the event $a$ in \L ukasiewicz logic with incomplete information.

The paper is organized as follows. In Section 2, we gather the basic notions and results on EMV-algebras. In Section 3 we show how EMV-algebras can appear from different algebraic structures, namely from naturally ordered monoids and generalized effect algebras, respectively. In Section 4, we present a definition of a state, describe extremal states, prove a variant of the Krein--Mil'man theorem for states. We establish some topological properties of the state spaces. We define pre-states and strong pre-states in Section 5 as a weaker form of states, and Section 6 describes Jordan signed measures which form a Dedekind complete $\ell$-group and which will be used for the integral representation of states by $\sigma$-additive regular probability measures in Section 7. Finally, in Section 7, we present a variant of the Horn--Tarski theorem showing that every state on an EMV-subalgebra can be extended to a state on the EMV-algebra.

\section{Basic Notions on EMV-algebras}

In the section, we gather the main notions and results on EMV-algebras. We start with MV-algebras introduced originally in \cite{Cha}.

Let $M=(M;\oplus,^*,0,1)$ be an MV-algebra, i.e. an algebra of type $
\langle 2,1,0,0 \rangle$ such that $(M;\oplus,0)$ is a commutative monoid with a neutral element $0$ and  for $x,y \in M$, we have
\begin{itemize}
\item[{\rm (i)}]  $x^{**}=x$;
\item[{\rm (ii)}] $x\oplus 1 = 1$;
\item[{\rm (iii)}] $x\oplus (x\oplus y^*)^* = y\oplus (y \oplus x^*)^*$.
\end{itemize}
We define another total binary operation $\odot$ by $x\odot y = (x^*\oplus y^*)^*$. Then $M$ is a distributive lattice for which $x\vee y = x\oplus (x\oplus y^*)^*$ and $x\wedge y = x \odot (x^*\oplus y)$.

We note that
$$x^* = \min\{z \in M\colon z\oplus x= 1\},\quad x \in M.
$$
If $a$ is a {\it Boolean element} of $M$ or an {\it idempotent}, i.e. $a\oplus a = a$ or equivalently, $a\vee a^*=1$, then the set $B(M)$ of Boolean elements of $M$ is a Boolean algebra that is also an MV-subalgebra of $M$. If $a$ is a Boolean element of $M$, then the interval $M_a:=[0,a]$ can be converted into an MV-algebra $([0,a];\oplus,^{*_a},0,a)$, where $x^{*_a}= a \odot x^*$ for each $x\in [0,a]$. Then we have
$$
x^{*_a}= \min\{z \in [0,a]\colon z\oplus x = a\}.
$$
In the paper, we will write also $\lambda_a(x):= x^{*_a}$, $x\in [0,a]$, i.e. \begin{equation}\label{eq:lambda}
\lambda_a(x)= \min\{z \in [0,a]\colon z\oplus x = a\},
\end{equation}
and $(M_a;\oplus,\lambda_a,0,a)$ is an MV-algebra.

A prototypical example of MV-algebras is creating from unital Abelian $\ell$-groups $(G,u)$, where $G$ is an Abelian $\ell$-group with a fixed strong unit $u$. On the interval $\Gamma(G,u):=[0,u]=\{g \in G\colon 0\le g \le u\}$ we define for $x,y \in [0,u]$,  $x\oplus y = (x+y)\wedge u$ and $x^*=u-x$. Then $\Gamma(G,u)=([0,u];\oplus, ^*,0,u)$ is an MV-algebra, and by Munduci's result, see \cite{CDM}, every MV-algebra is isomorphic to some $\Gamma(G,u)$.

Inspired by these properties of MV-algebras, in \cite{DvZa}, the authors introduced EMV-algebras as follows.

Let $(M;\oplus,0)$ be a commutative monoid with a neutral element $0$. All monoids in the paper are assumed to be commutative. An element $a \in M$ is said to be an {\it idempotent} if $a\oplus a=a$. We denote by $\mathcal I(M)$ the set of idempotent elements of $M$; clearly $0 \in \mathcal I(M)$, and if $a,b \in I(M)$, then $a\oplus b \in \mathcal I(M)$.

According to \cite{DvZa}, an {\it EMV-algebra} is an algebra $(M;\vee,\wedge,\oplus,0)$ of type $\langle 2,2,2,0 \rangle$ such that
\begin{enumerate}
\item[(i)] $(M;\oplus,0)$ is a commutative monoid with a neutral element $0$;
\item[(ii)] $(M;\vee,\wedge,0)$ is a distributive lattice with the bottom element $0$;
\item[(iii)] for each idempotent $a \in \mathcal I(M)$, the algebra $([0,a];\oplus, \lambda_a,0,a)$ is an $MV$-algebra;

\item[(iv)] for each $x \in M$, there is an idempotent $a$ of $M$ such that $x\le a$.
\end{enumerate}
We notify that according to (\ref{eq:lambda}), we have for each $a\in \mathcal I(M)$
$$
\lambda_a(x)=\min\{z \in [0,a]\mid z\oplus x = a\}, \quad x\in [0,a].
$$
We note that the existence of a top element in an EMV-algebra is not assumed, and if it exists, then $M=(M;\oplus, \lambda_1,0,1)$ is an MV-algebra. We underline that every MV-algebra forms an EMV-algebra, every generalized Boolean algebra (or equivalently a Boolean ring) is an EMV-algebra. In addition, the set of EMV-algebras is a variety, see \cite[Thm 3.11]{DvZa}.

Moreover, the operation $\odot$ can be defined as follows: Let $x,y \in M$ and let $x,y\le a\in \mathcal I(M)$. Then
$$x\odot y :=\lambda_a(\lambda_a(x)\oplus \lambda_a(y)).
$$

Let $x,y \le a,b$, where $a,b \in \mathcal I(M)$. By \cite[Lem 5.1]{DvZa}, we have
\begin{equation}\label{eq:x<y}
x\odot \lambda_a(y)=x\odot \lambda_a(x\wedge y) =x\odot \lambda_b(x\wedge y)=x\odot \lambda_b(y).
\end{equation}

An {\it ideal} of an EMV-algebra is a non-void subset $I$ of $M$ such that (i) if $x\le y \in I$, then $x\in I$, and (ii) if $x,y \in I$, then $x\oplus y$. An ideal is {\it maximal} if it is a proper ideal of $M$ which is not properly contained in another proper ideal of $M$. Nevertheless $M$ has not necessarily a top element, every $M\ne \{0\}$ has a maximal ideal, see \cite[Thm 5.6]{DvZa}. We denote by $\mathrm{MaxI}(M)$ the set of maximal ideals of $M$. The {\it radical} $\mbox{\rm Rad}(M)$ of $M$ is the intersection of all maximal ideals of $M$.

A subset $A\subseteq M$ is called an {\it EMV-subalgebra} of $M$ if $A$ is closed under $\vee$, $\wedge$, $\oplus$ and $0$ and, for each $b\in \mathcal I(M)\cap A$, the set $[0,b]_A:=[0,b]\cap A$ is a subalgebra of the MV-algebra $([0,b];\oplus,\lambda_b,0,b)$.

Let $(M_1;\vee,\wedge,\oplus,0)$ and $(M_2;\vee,\wedge,\oplus,0)$ be EMV-algebras. A map $f:M_1\to M_2$ is called an {\it EMV-homomorphism}
if $f$ preserves the operations $\vee$, $\wedge$, $\oplus$ and $0$, and for each $b\in\mathcal I(M_1)$ and for each $x\in [0,b]$, $f(\lambda_b(x))= \lambda_{f(b)}(f(x))$.

As we already said, it can happen that an EMV-algebra $M$ has no top element, however, it can be embedded into an MV-algebra $N$ as its maximal ideal as it was proved in the basic result \cite[Thm 5.21]{DvZa}:

\begin{theorem}\label{th:embed}{\rm [Basic Representation Theorem]}
Every EMV-algebra $M$ is either an MV-algebra or $M$ can be embedded into an MV-algebra $N$ as a maximal ideal of $N$ such that every element $x\in N$ either belongs to the image of the embedding of $M$, or it is a complement of some element $x_0$ belonging to the image of the embedding of $M$, i.e. $x=\lambda_1(x_0)$.
\end{theorem}
The MV-algebra $N$ from the latter theorem is said to be {\it representing} the EMV-algebra $M$. An analogous result for generalized Boolean algebras was established in \cite[Thm 2.2]{CoDa}.

For other unexplained notions and results, please consult with the papers \cite{DvZa, DvZa1}.

\section{EMV-algebras and Other Algebraic Structures}

In the section, we show how EMV-algebras can appear from different algebraic structures. The first result deals with naturally ordered monoids and the second one concerns generalized effect algebras.

A monoid $(M;\oplus,0)$ with a fixed partial order $\le$ is said to be {\it ordered} if $x\le y$ for $x,y \in M$ implies $x\oplus z\le y\oplus z$ for each $z\in M$, and we write $(M;\oplus,0,\le)$ for it. We say that an ordered monoid $(M;\oplus,0,\le)$ is {\it naturally ordered} if, for $x,y \in M$, $x\le y$ iff there is $z\in M$ with $x\oplus z = y$. We note that if such an order exists, is unique. Let $(M;\oplus,0)$ be a monoid. An element $a \in M$ is said to be an {\it idempotent} if $a\oplus a = a$. If $\mathcal I(M)$ is the set of idempotents of $M$, then (i) $0 \in \mathcal I(M)$, (ii) if $a,b\in \mathcal I(M)$, then $a\oplus b \in \mathcal I(M)$.

\begin{proposition}\label{pr:000}
Let $(M;\oplus,0,\le)$ be a naturally ordered commutative monoid with a neutral element $0$ satisfying the following conditions
\begin{itemize}
\item[{\rm (E1)}] for all $a \in \mathcal I(M)$, the algebra $([0,a];\oplus,\lambda_a,0,a)$ is an MV-algebra;
\item[{\rm (E2)}] for each $x \in M$, there is an idempotent $a \in M$ such that $x\le a$.
\end{itemize}
Then $\le$ is a lattice order and $M$ is a distributive lattice with respect to $\le$ with the least element $0$. Moreover, $(M;\vee,\wedge,\oplus,0)$ is an EMV-algebra.
\end{proposition}

\begin{proof}
We start with a note that according to (\ref{eq:lambda}), we have that the element $\lambda_a(x)=\min\{z\in[0,a]\colon x\oplus z=a\} $
exists in $M$ for all $x\in [0,a]$.

(i) First we show that if $x\le a\le c$ and $a,c \in \mathcal I(M)$, then $\lambda_c(x)=\lambda_a(x)\oplus \lambda_c(a)$. Since $[0,c]$ is an MV-algebra, using Mundici's result, see e.g. \cite{CDM}, there is an Abelian unital $\ell$-group $(G,u_c)$  such that $[0,c]\cong \Gamma(G,u_c)$ and without loss of generality, we can assume that $[0,c]=\Gamma(G,u_c)$ and $c=u_c$. Since $[0,a]\subseteq [0,c]$, and $[0,a]$ is also an MV-algebra, then $[0,a]=\Gamma(G(a),a)$, where $G(a)$ is an Abelian $\ell$-group such that $G(a)=\{g \in G\mid \exists\, n\in \mathbb N,\ g\le na\}$, $G(a)$ is an $\ell$-subgroup of the $\ell$-group $G$.  If $x\le a$, then $\lambda_a(x)=a-x$, $\lambda_c(x)=c-x$, and $c=x\oplus \lambda_c(x)= x+(c-x)$, $a=x \oplus \lambda_a(x)=a+(a-c)$, where $-$ and $+$ is the group subtraction and the group addition taken from $G(a)$ and $G$. Then $\lambda_a(x)\oplus \lambda_c(a)=(a-x)\oplus (c-a)= ((a-x)+(c-a))\wedge c= c-x=\lambda_c(x)$.

In addition, $c=a\oplus \lambda_c(a)$ implies that $a$ is a Boolean element of the MV-algebra $[0,c]$, hence, $\lambda_c(a)$ is an idempotent of $[0,c]$, consequently, it is an idempotent of $M$.

(ii) Let $x,y \le a,b$, where $a,b\in \mathcal I(M)$. There is an idempotent $c\in \mathcal I(M)$ with $a,b\le c$. In the MV-algebra $[0,a]$, there is a distributive lattice structure with respect to $\vee_a$ and $\wedge_a$. Analogously, if $x,y \le a \le c\in \mathcal I(M)$, we have a distributive lattice structure $\vee_c$ and $\wedge_c$. We define also $x\odot_a y=\lambda_a(\lambda_a(x)\oplus \lambda_a(y))$. Similarly, let $a\le c \in \mathcal I(M)$, and $x\odot_c y=\lambda_c(\lambda_c(x)\oplus \lambda_c(y))$.

Using (i), we have
\begin{eqnarray*}
\lambda_c(\lambda_c(x)\oplus \lambda_c(y)) &=& \lambda_c\big(\lambda_a(x)\oplus \lambda_c(a)\oplus \lambda_a(y)\oplus \lambda_c(a)\big)\\
&=& \lambda_c\big(\lambda_a(x)\oplus \lambda_a(y)\oplus \lambda_c(a)\big)\\
&=& \lambda_c\big(\lambda_a(x)\oplus \lambda_a(y)\big)\odot_{_c} \lambda_c(\lambda_c(a))=\lambda_c\big(\lambda_a(x)\oplus \lambda_a(y)\big)\odot_{_c} a\\
	&=& \lambda_c\big(\lambda_a(x)\oplus \lambda_a(y)\big)\wedge_c a=\Big(\lambda_a\big(\lambda_a(x)\oplus \lambda_a(y)\big)\oplus \lambda_c(a)\Big)\wedge_c a \\
	&=& \Big(\lambda_a\big(\lambda_a(x)\oplus \lambda_a(y)\big)\vee_c \lambda_c(a)\Big)\wedge_c a=\lambda_a\big(\lambda_a(x)\oplus \lambda_a(y)\big)\wedge_c a \\
	&=& \lambda_a\big(\lambda_a(x)\oplus \lambda_a(y)\big).
\end{eqnarray*}

In a similar way we have $x \odot_b y= x\odot_c y$, that is $x \odot_a y= x\odot_c y= x\odot_b y$ which proves that we can define $x\odot y$ as $x\odot y =x\odot_a y$ whenever $x,y \le a\in \mathcal I(M)$.

(iii) Let $x\le y\le a,b\le c$ for $a,b,c \in \mathcal I(M)$. Then by (i) and using the distributivity of $\oplus$ with respect to $\vee$ and $\wedge$ in any MV-algebra, we have

$$
y \odot \lambda_c(x)= y \odot (\lambda_a(x)\oplus \lambda_c(a))=y \odot (\lambda_a(x)\vee_c \lambda_c(a))= (y \odot \lambda_a(x)) \vee_c (y\odot\lambda_c(a))
$$
and
$$y\odot\lambda_c(a) \le y \odot \lambda_c(y)=0$$
because  for $y\le a\le c$ we have $\lambda_c(a)\le \lambda_c(y)$.  This implies $y\odot  \lambda_a(x)=y\odot \lambda_c(x)$. In the same way we have $y\odot  \lambda_b(x)=y\odot \lambda_c(x)$ establishing
$$y\odot  \lambda_a(x)=y\odot \lambda_b(x) \quad \text{ if } \quad x\le y.$$

Now let $x,y \le a,b$ for some $a,b \in \mathcal I(M)$. Then $x \odot \lambda_a(x\wedge y)= x \odot (\lambda_a(x)\vee_a \lambda_a(y))= (x\odot \lambda_a(x))\vee_a (x\odot \lambda_a(y))= x \odot \lambda_a(y)= x \odot \lambda_b(x\wedge y) = x\odot \lambda_b(y)$, i.e.
$$y\odot  \lambda_a(x)=y\odot \lambda_b(x) \quad \text{ if } \quad x, y \le a,b.$$

(iv) Finally,  let $x,y \le a,b$, where $a,b\in \mathcal I(M)$.
We have for the suprema $x\vee_a y$ taken in the MV-algebra $[0,a]$, $x\vee_a y= (x\odot_a \lambda_a(y))\oplus y= (x\odot \lambda_a(y))\oplus y$ and for the supremum  $x\vee_b y$ taken in the MV-algebra $[0,b]$, $x\vee_a y= (x\odot_b \lambda_b(y))\oplus y$. Then by the latter equality, we have
$x\vee_a y = (x\odot_a \lambda_a(y))\oplus y = (x\odot \lambda_a(y))\oplus y= (x\odot \lambda_b(y))\oplus y = (x\odot_b \lambda_a(y))\oplus y= x\vee_b y$.

In a dual way, if $x,y \le a,b\le c$, where $a,b,c\in \mathcal I(M)$, then
we have $x\wedge_a y =\lambda_a(\lambda_a(x)\vee\lambda_a(y))$ and $x\wedge_c y =\lambda_c(\lambda_c(x)\vee\lambda_c(y))$. Since $x\wedge_a y \in [0,a]\subseteq [0,c]$, we have $x\wedge_a y \le x\wedge_c y$. On the other hand, from  $x\wedge_c y \le x,y \le a$, we get $x\wedge_c y \le x\wedge_a x$, so that $x\wedge_a y = x\wedge_c y = x\wedge_b y$. Whence, we can define $\vee$ and $\wedge$ in the whole $M$ as follows $x\vee y= x\vee_a y$ and $x\wedge y = x\wedge_a y$ whenever $x,y \le a\in \mathcal I(M)$.  Since $\wedge_a$ and $\vee_a$ are distributive in the    MV-algebra $[0,a]$, we have $(M;\vee,\wedge,0)$ is a distributive lattice with the least element $0$.

In addition, for the original ordering $\le$ on $M$, we have $x\le y$ iff $x\vee y =y$.
\end{proof}

We remind that according to \cite{DvPu}, an algebra $(E;+,0)$, where $+$ is a partial operation on $E$, is said to be a {\it generalized effect algebra} (GEA for short) if, for all $x,y,z \in E$, we have
\begin{itemize}
\item[(i)] if $x+y$ is defined, then $y+x$ is defined and $x+y=y+x$;
\item[(ii)] if $x+y$ and $(x+y)+z$ are defined, then $y+z$ and $x+(y+z)$ are defined an $(x+y)+z= x+(y+z)$;
\item[(iii)] $x+0=x$;
\item[(iv)] if $x+y=x+z$, then $y=z$;
\item[(v)] if $x+y=0$, then $x=y=0$.
\end{itemize}

If a GEA $E$ has a top element, $(E;+,0,1)$ is said to be an {\it effect algebra}. The original axioms of effect algebras are as follows, see \cite{FoBe}:
\begin{itemize}
\item[(i)] if $x+y$ is defined, then $y+x$ is defined and $x+y=y+x$;
\item[(ii)] if $x+y$ and $(x+y)+z$ are defined, then $y+z$ and $x+(y+z)$ are defined and $(x+y)+z= x+(y+z)$;
\item[(iii)] for every $x\in E$, there is a unique element $x'\in E$ (called a {\it orthosuplement} of $x$) such that $x+x'=1$;
\item[(iv)] if $1+x$ is defined, then $x=0$.
\end{itemize}

We note that a non-void subset $I$ of a GEA $(E;+,0)$ is said to be a {\it GEA-ideal} (simply an ideal) if (i) $x\le y \in I$ implies $x \in I$, and (ii) if $x,y \in I$ and $x+y$ is defined in $E$, then $x+y \in I$. In every GEA $(E;\oplus,0)$ we can define an order $\le :=\le_E$ by $x\le y$ iff there is $z \in E$ such that $x+z=y$; if such an element exists, it is unique and we write also $z=y-x$. We call $\le_E$ also as a {\it GEA-order} induced from the GEA $E$. If $E$ is under $\le=\le_E$ a lattice, we call it a lattice GEA.

A GEA $(E;\oplus,0)$ satisfies the {\it Riesz Decomposition Property} if, given elements $x_1,x_2, y_1,y_2\in E$ such that $x_1+x_2=y_1+y_2$, there are four elements $c_{11}, c_{12}, c_{21}, c_{22} \in E$ such that $x_1=c_{11}+c_{12}$, $x_2 = c_{21} + c_{22}$, $y_1= c_{11} +c_{21}$ and $y_2= c_{12} + c_{22}$. In addition, let $E$ satisfy RDP. If $x\le u+v$ for $x,u,v\in E$, there are $u_1,v_1\in E$ with $u_1\le u$ and $v_1\le v$ such that $u=u_1+v_1$.

If $(E;+,0)$ is a GEA and $x\in M$ is a fixed element, then $E_x=([0,x];+,0,x)$, where $[0,x]=\{y\in E\colon 0\le y \le x\}$, is an effect algebra; the orthosuplement of $y\in [0,x]$  is the element $y^{'x}=x-y$.

It is well-known, see e.g. \cite[Thm 1.8.12]{DvPu}, that if an effect algebra $E$ is a lattice effect algebra satisfying RDP, we can define a binary operation $x\oplus y=x+(y\wedge x')$ for all $x,y\in E$ such that $(E;\oplus,',0,1)$ is an MV-algebra. In particular, if $E$ is a lattice GEA with RDP, then every $([0,a];\oplus_a,^{'a},0,a)$ ($a\in E$), where
\begin{equation}\label{eq:RDP}
x\oplus_a y:= x+ (y\wedge x^{'a})=x+(y\wedge (a-x)),\quad x,y \in [0,a],
\end{equation}
is an MV-algebra.

\begin{proposition}\label{pr:state1}
Let $(M;\vee,\wedge,\oplus,0)$ be an EMV-algebra. We define a partial operation $+$ on $M$ in such a way that $x+y$ is defined iff $x\odot y = 0$, and in such a case, $x+y:=x\oplus y$. Then $(M;+,0)$ is a generalized effect algebra satisfying the Riesz Decomposition Property where the order on $M$ induced by $\wedge,\vee$ and the GEA-order induced from  $(M;+,0)$ coincide.
\end{proposition}

\begin{proof}
First we note that $x\odot y=0$ iff there is an idempotent $a \in \mathcal I(M)$ such that $x\le \lambda_a(y)$ with $x,y \le a$, consequently, iff $x\le \lambda_a(y)$ for each idempotent $a \in \mathcal I(M)$ such that $x,y \le a$.

The commutativity is evident. To prove the associativity, assume $x+y$ and $(x+y)+z$ are defined. Choose an idempotent $a \in \mathcal I(M)$ with $x,y,z\le a$.
Therefore, $z\le \lambda_a(x+y)=\lambda_a(x\oplus y)=\lambda_a(x)\odot \lambda_a(y)\le \lambda_a(x),\lambda_a(y)$. Hence, $y+z$ is defined in $M$ and $y+z = y\oplus z\le y \oplus (\lambda_a(x)\odot \lambda_a(y))=y\vee \lambda_a(x)=\lambda_a(x)$. Therefore, $x+(y+z)=(x+y)+z$.

Clearly, $x+0$ is defined for each $x\in M$ and $x+0=x$.

Let $x+y=x+z$, and let $x,y,z \le a \in \mathcal I(M)$.
Then
$$ a= (x+y)\oplus \lambda_a(x+y)= (x+y)+\lambda_a(x+y)=(x+z)+\lambda_a(x+y).$$
Using the cancelation law holding for $+$ in the MV-algebra $[0,a]$, we have $y=z$.


Finally, let $x+y = 0$. Then $x,y \le x\oplus y = x+y=0$.

Let $\le$ be the order generated by the lattice operations $\vee$ and $\wedge$ defined in the EMV-algebra $M$.  We set $x\preceq y$ iff there is $z\in M$ such that $x+z=y$. Take an idempotent $a\in \mathcal I(M)$ such that $x,y,z\le a$. Then $x+z=x\oplus z =y$ which means also $x\le y$. Now let $x\le y$, then $y = x\vee y= x\oplus (y \odot \lambda_a(x))= x +(y \odot \lambda_a(x))$, i.e. $x\preceq y$ and $\preceq = \le$, so that $\preceq$ is a lattice order, and $M$ is a distributive generalized effect algebra.

If $x_1,x_2,y_1,y_2\in M$ satisfy $x_1+x_2=y_1+y_2$, there is an idempotent $a \in M$ such that $x_1,x_2,y_1,y_2\le a$. Then $x_1,x_2,y_1,y_2 \in [0,a]$, and since $([0,a];\oplus,\lambda_a,0,a)$ is an MV-algebra, it satisfies RDP and the partial addition $+$ coincides with the partial sum induced from the MV-algebra $[0,a]$. Moreover, the order in the MV-algebra $[0,a]$ coincides with the original order on $M$ restricted to $[0,a]$, therefore, RDP holds in $(M;\oplus,0)$.
\end{proof}

Now we show how from a GEA we can derive an EMV-algebra.  Let $(E;+,0)$ be a lattice GEA with RDP. An element $a\in E$ is said to be {\it Boolean}, if for each $b\in E$ with $b\ge a$, we have $a\oplus_b a=a$, where $\oplus_b$ is defined by (\ref{eq:RDP}).

\begin{proposition}\label{pr:EMV-GEA}
Let $(E;+,0)$ be a lattice GEA satisfying RDP and let, for every $x \in E$, there be a Boolean element $a\in E$ such that $x\le a$. Then there is a binary operation $\oplus$ on $E$ such that $(E;\vee,\wedge,\oplus,0)$ is an EMV-algebra, and an element $a\in E$ is Boolean if and only if $a\oplus a = a$. Moreover, the partial addition derived from the EMV-algebra $(E;\vee,\wedge,\oplus,0)$ coincides with the original $+$ in the GEA $E$.
\end{proposition}

\begin{proof}
We note that if $a \in E$, then in the effect algebra $E_a=([0,a];+,0,1)$ we have for $x,y \in [0,a]$, $x\le_E y $ iff $x\le_{E_a} y$. Due to \cite[Thm 3.2]{Dvu2}, an element $a \in E$ with $a\le b \in E$ is Boolean iff $a\wedge (b-a)=0$. In addition, if $a$ is a Boolean element of $E$ and $a\le b \in E$, then for each $x,y \in [0,b]$ with $x+y \in [0,b]$, due to \cite[Prop 2.5]{Dvu2}, we have $(x+y)\wedge a= (x\wedge a)+(y\wedge a)$.

If a GEA $E$ has a top element $1$, it is a Boolean element, and we set $x\oplus y = x +(y\wedge x')$, $x,y \in E$, so that $(E;\oplus,',0,1)$ is an MV-algebra, and $(E;\vee,\wedge,\oplus,0)$ is an EMV-algebra.

Now let $E$ have no top element. The binary operation $\oplus$ is defined as follows. Given $x,y\in E$, there is a Boolean element $a\in E$ such that $x,y \le a$. If there is another Boolean element $b\in E$ with $x,y \le b$, there is a third Boolean element $c\in E$ such that $x,y\le a,b \le c$. Then we have $a\oplus_a a= a= a\oplus_c a$, and $x\oplus_a y= x+(y\wedge (a-x))$ and $x\oplus_c y= x+(y\wedge (c-x))$. In addition, $c = (c-x)+ x$ so that $a=c\wedge a= ((c-x)+x)\wedge a= ((c-x)\wedge a)+(x\wedge a)$ which yields
$a-x= (c-x)\wedge a$. Hence,
\begin{eqnarray*}
x\oplus_c y&=& x+(y\wedge (c-x))\\
&=& x+((y\wedge a)\wedge (c-x)\wedge a)\\
&=& x+(y\wedge (a-x))\\
&=& x\oplus_a y.
\end{eqnarray*}
In the same way we have $x\oplus_b y=x\oplus_c y$, which shows that we can define unambiguously $\oplus$ by $x\oplus y=x\oplus_a y$ whenever $a$ is a Boolean element of $E$ such that $x,y \le a$.

It is clear that if $a$ is a Boolean element of $E$, then $a\oplus a=a=a\oplus_a a$. Conversely, let $a\oplus a = a$, then for each $b\ge a$, we have $a=a\oplus_b a$, so that $a$ is a Boolean element of $E$.

Consequently, $(E;\vee,\wedge,\oplus,0)$ is an EMV-algebra, and if $x,y\in E$ and $x,y\le a$, where $a$ is a Boolean element, then $x\odot y = 0$ iff $x\le \lambda_a(y)=a-y$, so that $x+y\le a$ and $x+y$ is defined in $E$ as well as in $[0,a]$. Conversely, let $x+y$ be defined in $E$ and let $a$ be a Boolean element such that $x+y\le a$. Then $x\le a-y=\lambda_a(y)$, so that $x\odot y =0$.
\end{proof}

\section{States on EMV-algebras}

In the present section, we introduce states. A stronger notion of states, state-morphisms, were introduced in \cite{DvZa} and used in \cite{DvZa1} for establishing the Loomis--Sikorski theorem for $\sigma$-complete EMV-algebras. A state, an analogue of a finitely additive measure, is defined here for EMV-algebras even if they have not necessarily a top element. If an EMV-algebra possesses a top element, the state is the same as that for MV-algebras. We show that state-morphisms are only extremal states. We establish that all state-morphisms generate in some sense all states, the Krein--Mil'man-type representation. In addition, some topological properties of the state space are investigated.

We note that according to \cite{DvZa}, a mapping $s: M\to [0,1]$ such that $s$ is an EMV-homomorphism from $M$ into the MV-algebra of the real interval $[0,1]$ is said to be a {\it state-morphism} if there is an element $x \in M$ with $s(x)=1$. We denote by $\mathcal{SM}(M)$ the set of all state-morphisms on $M$. If $M\ne \{0\}$, $M$ possesses at least one state-morphism, see \cite[Thm 4.2]{DvZa}. The basic properties of state-morphisms were established in \cite[Prop 4.1]{DvZa}:

\begin{proposition}\label{pr:state2}
Let $s$ be a state-morphism on an EMV-algebra $M$. Then

\begin{itemize}
\item[{\rm (i)}] $s(0)=0$;
\item[{\rm (ii)}] $s(a)\in \{0,1\}$ for each idempotent $a \in M$;
\item[{\rm (iii)}] if $x\le y$, then $s(x)\le s(y)$;
\item[{\rm (iv)}] $s(\lambda_a(x))=   s(a)-s(x)$ for each $x\in [0,a]$, $a \in \mathcal I(M)$;
\item[{\rm (v)}] $\Ker(s)$ is a proper ideal of $M$,  where $\Ker(s)=\{x\in M\colon s(x)=0\}$.
\end{itemize}
\end{proposition}

We recall that a partial operation $+$, that is commutative and associative, on an EMV-algebra $M$ was defined in Proposition \ref{pr:state1}.
We say that a mapping $s:M \to [0,1]$ is a {\it state} on $M$ if (i) $s(x+ y)=s(x)+s(y)$ whenever $x+y$ is defined in $M$, and (ii) there is an element $a\in M$ such that $s(a)=1$.

A state is an analogue of a finitely additive probability measure, and states for MV-algebras were defined in \cite{Mun2} as averaging the truth-value in \L ukasiewicz logic.

We denote by $\mathcal S(M)$ the set of states on $M$. The basic properties of states are as follows:

\begin{proposition}\label{pr:state3}
Let $s$ be a state on an EMV-algebra $M$. For all $x,y \in M$, we have
\begin{enumerate}
\item[{\rm (i)}] $s(0)=0$;
\item[{\rm (ii)}] if $x\le y\le a \in \mathcal I(M)$, then $s(x)\le s(y)$ and $s(y\odot \lambda_a(x))=s(y)-s(x)$;
\item[{\rm (iii)}] $s(x\vee y)+s(x\wedge y)=s(x)+s(y)$;
\item[{\rm (iv)}] $s(x\oplus y) + s(x\odot y)=s(x)+s(y)$;
\item[{\rm (v)}]  $\Ker(s)=\{x\in M \colon s(x)=0\}$ is an ideal of EMV-algebra $M$ as well as an GEA-ideal of the GEA $(M;+,0)$;
\item[{\rm (vi)}] if $s_1,s_2 \in \mathcal S(M)$ and $\lambda \in [0,1]$ is a real number, then the convex combination $s= \lambda s_1 +(1-\lambda)s_2$ of states $s_1,s_2$ is a state on $M$.
\item[{\rm (vii)}] If we define a mapping $\hat s$ on the quotient EMV-algebra $I/\Ker(s)$ by $\hat s(x/\Ker(s)):= s(x)$, $(x\in M)$, then $\hat s$ is a state on $M/\Ker(s)$, and $M/\Ker(s)$ has a top element.
\end{enumerate}
\end{proposition}

\begin{proof}
(i) Since $0=0+0$, we have $s(0)=0$.

(ii) If $x\le y$, then there is an element $z \in M$ such that $x+z=y$, so that $s(z)=s(y)-s(x)\ge 0$. In addition, since $x+(y\odot \lambda_a(x))$ exists in $M$ and $x+(y\odot \lambda_a(x))=x\oplus (y\odot \lambda_a(x))=y$, we conclude the result.

(iii) Given $x,y \in M$, there is an idempotent $a\in M$ with $x,y \le a$. Then in the MV-algebra $[0,a]$, we have $(x\vee y)\odot \lambda_a(x)=x\odot \lambda_a(x\wedge y)$ which by (ii) gives (iii).

(iv) If $x,y \le a \in \mathcal I(M)$, using \cite[Prop 1.25]{GeIo}, we have $x= ((x\oplus y)\odot \lambda_a(y))\oplus (x\odot y)$. Hence, (ii) implies the result.

(v) From (ii), we have that $\Ker(s)$ is a down-set. From the identity  $x\oplus y = x +((x\oplus y)\odot \lambda_a(x))= x +(y\wedge \lambda_a(x))$, we get $\Ker(s)$ is closed under $\oplus$.

(vi) It is clear.

(vii) Let $x,y \in M$. Then $x/\Ker(s)=y/\Ker(s)$ iff $s(x)=s(x\wedge y)=s(y)$. Therefore, $\hat s$ is correctly defined. For each $x\in M$, let $[x]=x/\Ker(s)$. Given $x,y \in M$, there is an idempotent $a\in \mathcal I(M)$ such that $x,y \le a$ and $s(a)=1$. Hence, $\lambda_a(x)/\Ker(s) = \lambda_{[a]}([x])$. Assume that $[x]\le [\lambda_a(y)]$. For $x_0= x\wedge \lambda_a(y)$, we have $x_0 \le \lambda_a(y)$ and $[x_0]=[x\wedge \lambda_a(x)]=[x]\wedge [\lambda_a(y)]=[x]\le [\lambda_a(y)]=\lambda_{[a]}([y])$. Then
\begin{eqnarray*}
\hat s([x]+[y])&=& \hat s([x\oplus y])= \hat s([x_0+y])= s(x_0+y)= s(x_0)+s(y)\\
&=& \hat s([x_0])+\hat s([y])= \hat s([x])+\hat s([y]).
\end{eqnarray*}
In addition, $\hat s([a])=s(a)=1$, so that $\hat s$ is a state on $M/\Ker(s)$. Since there is an element $a\in M$ such that $s(a)=1$, and $s(x)\le 1$, we have that the element $[a]$ is the top element for $M/\Ker(s)$.
\end{proof}

We say that a state $s$ is {\it extremal} if from  $s= \lambda s_1 +(1-\lambda)s_2$, for $\lambda \in (0,1)$ and $s_1,s_2 \in \mathcal S(M)$, we have $s=s_1=s_2$. We denote by $\partial \mathcal S(M)$ the set of extremal states on $M$. In what follows, we show that $\partial \mathcal S(M) = \mathcal{SM}(M)$.

\begin{proposition}\label{pr:state4}
Let $M$ be an EMV-algebra. Every state-morphism on $M$ is a state.
Let  $s$ be a state on $M$. The following statements are equivalent:
\begin{enumerate}
\item[{\rm (i)}] $s$ is a state-morphism.
\item[{\rm (ii)}] $s(x\wedge y)=\min\{s(x),s(y)\}$, $x,y \in M$.
\item[{\rm (iii)}] $s(x\vee y)= \max\{s(x),s(y)\}$, $x,y \in M$.
\item[{\rm (iv)}] $s(x\oplus y)= \min\{s(x)+s(y),1\}$, $x,y \in M$.
\end{enumerate}
\end{proposition}

\begin{proof}
Let $s$ be a state-morphism on $M$. For $x,y \in M$, we can find an idempotent $a\in M$ such that $x,y\le a$ and $s(a)=1$. Since $s$ is an EMV-homomorphism, we have $s(x\odot y)=s(x)\odot s(y)$. Therefore, if $x\odot y=0$, then $s(x)\odot s(y)=\max\{s(x)+s(y)-1,0\}=0$ which yields $s(x+y)=s(x\oplus y)=s(x)\oplus s(y)=\min\{s(x)+s(y),1\}=s(x)+s(y)$ which shows $s$ is a state on $M$.

Now let $s$ be an arbitrary state on $M$. Similarly, for all $x,y \in M$ there is $a\in M$ such that $s(a)=1$.
(i) $\Rightarrow$ (ii). Since $s$ is an EMV-homomorphism, (ii) holds.

(ii) $\Leftrightarrow$ (iii). It follows from the equalities $\lambda_a(x\vee y)=\lambda_a(x)\wedge \lambda_a(y)$ and $\lambda_a(x\wedge y)=\lambda_a(x)\vee \lambda_a(y)$.

(ii) $\Rightarrow$ (iv).
We have $x\oplus y = x +((x\oplus y)\odot \lambda_a(x))= x+ (\lambda_a(x)\wedge y)$. Then $s(x \oplus y) = s(x) + s(\lambda_a(x)\wedge y) = s(x) + \min\{1 - s(x),s(y)\} = \min\{s(x) + s(y), 1\}= s(x)\oplus s(y)$.

(iv) $\Rightarrow$ (i). First we show that $s(x\odot y)=s(x)\odot s(y)$. Indeed, $s(x\odot y)=s(\lambda_a(\lambda_a(x)\oplus \lambda_a(y)))= s(a)- s(\lambda_a(x)\oplus \lambda_a(y)) = 1- s(\lambda_a(x))\oplus s(\lambda_a(y))= s(x)\odot s(y)$. Therefore, $s(x\wedge y)=s(x\odot (\lambda_a(x)\oplus y))= s(x)\odot ((1-s(x))\oplus s(y)) =\min\{s(x),s(y)\}$. Similarly, $s(x\vee y)=\max\{s(x),s(y)\}$. Hence, $s$ preserves $\oplus,\vee,\wedge, \odot$. In addition, $s(\lambda_a(x))= 1 - s(x)= \lambda_{s(a)}(s(x))$, i.e. $s$ is an EMV-homomorphism.
\end{proof}

Since every state-morphism is a state on $M$, if $M\ne\{0\}$, $M$ possesses at least one state because it has at least one state-morphism, see  \cite[Thm 4.2, Thm 5.6]{DvZa}.

\begin{proposition}\label{pr:state5}
A state $s$ on an EMV-algebra $M$ is a state-morphism if and only if $\Ker(s)$ is a maximal ideal of $M$.
\end{proposition}

\begin{proof}
If $s$ is a state-morphism, by \cite[Thm 4.2(ii)]{DvZa}, $\Ker(s)$ is a maximal ideal of $M$.  Conversely, let $\Ker(s)$ be a maximal ideal of $M$. Take $x,y \in M$; there is an idempotent $a\in M$ such that $x,y \le a$ and $s(a)=1$. Then in the MV-algebra $[0,a]$, we have $(x\odot \lambda_a(y))\wedge (y\odot \lambda_a(x))=0$. Every maximal ideal is prime, so that $s(x\odot \lambda_a(y))=0$ or $s(y\odot \lambda_a(x))=0$. In the first case we have $0=s(x\odot \lambda_a(y))= s(x\odot \lambda_a(x\wedge y))= s(x)-s(x\wedge y)$, where we have used a fact $x\odot \lambda_a(y)=x\odot \lambda_a(x\wedge y)$, see (\ref{eq:x<y}), and in the second case, we have $s(y)=s(x\wedge y)$, i.e. $s(x\wedge y) = \min\{s(x),s(y)\}$, which by Proposition \ref{pr:state4} means $s$ is a state-morphism.
\end{proof}

We note that in \cite[Thm 4.2]{DvZa}, the following important characterization of state-morphisms by maximal ideals was established.

\begin{theorem}\label{th:state6}
{\rm (1)} If $I$ is a maximal ideal of an EMV-algebra $M$, then there is a unique state-morphism $s$ on $M$ such that $\Ker(s)=I$.

{\rm(2)} If $s_1$ and $s_2$ are state-morphisms, then $s_1=s_2$ if and only if $\Ker(s_1)=\Ker(s_2)$.
\end{theorem}

\begin{proposition}\label{pr:state7}
A state $s$ on an EMV-algebra $M$ is extremal if and only if $s$ is a state-morphism.
\end{proposition}

\begin{proof}
Let $s$ be an extremal state on $M$ and define a state $\hat s$ on $M/\Ker(s)$ by Proposition \ref{pr:state3}(vii). We assert that $\hat s$ is an extremal state on $M/\Ker(s)$. Indeed, let $m_1, m_2$ be states on $m/\Ker(s)$ and $\lambda \in (0,1)$ such that $\hat s = \lambda m_1 +(1-\lambda) m_2$. There exist two states $s_1$ and $s_2$ on $M$ such that $s_i(x) = m_i([x])$ for each $x \in M$ and $i=1,2$. Then $s_i(x+y)= s_i(x)+s_i(y)$. For $m_i$ there is an element $a \in M$ such that $m_i([a])=1$, so that $s_i$ is a state on $M$. In addition, $s= \lambda s_1 +(1-\lambda)s_2$ which implies $s=s_1 = s_2$, so that $m_1=m_2$.

Due to Proposition \ref{pr:state3}(vii), $M/\Ker(s)$ is an EMV-algebra with a top element, alias, $M/\Ker(s)$ is an MV-algebra and $\hat s$ is an extremal state on the MV-algebra $M/\Ker(s)$. Hence, by \cite[Thm
6.1.30]{DvPu}, $\hat s$ is a state-morphism, consequently so is $s$ on $M$.

Conversely, let $s$ be a state-morphism on $M$ and let $s=\lambda s_1 +(1-\lambda)s_2$ for some $s_1,s_2 \in \mathcal S(M)$ and $\lambda \in (0,1)$. Then $\Ker(s) = \Ker(s_1) \cap \Ker(s_2)$ and the maximality of $\Ker(s)$ entails $\Ker(s)=\Ker(s_1)=\Ker(s_2)$, so that Proposition \ref{pr:state5} says that $s_1$ and $s_2$ are state-morphisms and by Theorem \ref{th:state6}, $s=s_1 = s_1$. Consequently, $s$ is an extremal state on $M$.
\end{proof}

\begin{theorem}\label{th:state8}
Let $M$ be an EMV-algebra. Then $\partial \mathcal S(M)=\mathcal{SM}(M)$.
\end{theorem}

\begin{proof}
If $M=\{0\}$, then $\mathcal S(M)=\emptyset = \mathcal{SM}(M)$. If $M \ne \{0\}$, the result is a direct corollary of Proposition \ref{pr:state7}.
\end{proof}

We say that a net $\{s_\alpha\}_\alpha$ of states on $M$ {\it converges weakly} to a state $s$ on $M$, and we write $\{s_\alpha\}_\alpha\stackrel{w} \rightarrow s$, if $\lim_\alpha s_\alpha(a)=s(a)$ for each $a\in M$. Hence, $\mathcal{S}(M)$ is a subset of $[0,1]^M$ and if we endow $[0,1]^M$ with the product topology which is a compact Hausdorff space, we see that the weak topology, which is in fact a relative topology (or a subspace topology) of the product topology of $[0,1]^M$, yields a non-empty Hausdorff topological space whenever $M\ne \{0\}$; if $M=\{0\}$, the set $\mathcal{S}(M)$ is empty. In addition, the system of subsets of $\mathcal{S}(M)$ of the form $S(x)_{\alpha,\beta}=\{s \in \mathcal{S}(M) \mid \alpha<s(x)<\beta\}$, where $x\in M$ and $\alpha < \beta$ are real numbers, forms a subbase of the weak topology of states.

The weak topology can be defined also for the set of state-morphisms in the same way as it was done for states. Due to Proposition \ref{pr:state4}, $\mathcal{SM}(M)$ is a closed subset of $\mathcal S(M)$, and $\mathcal{SM}(M)$ is also a Hausdorff space. The spaces $\mathcal S(M)$ and $\mathcal{SM}(M)$ are not necessarily compact sets because if, for a net $\{s_\alpha\}$ of states (state-morphisms), there is $s(x)=\lim_\alpha s_\alpha(x)$, $x \in M$, then $s$ preserves $+$ ($\oplus,\wedge,\vee$), but there is no guarantee that there is an element $x\in M$ with $s(x)=1$ as the following example shows.

\begin{example}\label{ex:state9}
Let $\mathcal T$ be the system of all finite subsets of the set $\mathbb N$ of natural numbers. Then $\mathcal{SM}(\mathcal T)=\{s_n \colon n \in \mathbb N\}$, where $s_n(A) = \chi_A(n)$, $A \in \mathcal T$. Given $A \in \mathcal T$, there is $s(A)=\lim_ns_n(A) = 0$, but $s$ is not a state on $\mathcal T$.
\end{example}

In the following result we show conditions when the spaces $\mathcal S(M)$ and $\mathcal{SM}(M)$ are compact in the weak topology of states.

\begin{proposition}\label{pr:state9}
Let $M$ be an EMV-algebra. Then the following statements are equivalent.
\begin{itemize}
\item[{\rm (i)}] $M$ has a top element.
\item[{\rm (ii)}] $\mathcal S(M)$ is compact.
\item[{\rm (iii)}] $\mathcal{SM}(M)$ is compact.
\end{itemize}
\end{proposition}

\begin{proof}
If $M=\{0\}$, then $0$ is the top element and $\mathcal S(M)=\emptyset = \mathcal{SM}(M)$.  Thus let $M\ne \{0\}$.

(i) $\Rightarrow$ (ii),(iii).  If $1$ is the top element of $M$, then $s(1)=1$ for each state $s$ on $M$. Therefore, $\mathcal S(M)$ and $\mathcal{SM}(M)$ are closed in the product topology on $[0,1]^M$, so that both sets are compact in the weak topology.

(ii) $\Rightarrow$ (iii). If $\mathcal S(M)$ is compact, then $\mathcal{SM}(M)$, which is a closed subset of $\mathcal S(M)$, has to be compact, too.

(iii) $\Rightarrow$ (i). Given $x \in M$, let $S(x)=\{s \in \mathcal{SM}(M) \colon s(x)>0\}$. Then each $S(x)$ is an open set of $\mathcal{SM}(M)$. Given $s \in \mathcal{SM}(M)$, there is an idempotent $a \in M$ such that $s(a)=1$, so that $s \in S(a)$ which means that $\{S(a)\colon a \in \mathcal I(M)\}$ is an open cover of $\mathcal{SM}(M)$. The compactness of $\mathcal{SM}(M)$ entails there are elements $a_1,\ldots,a_n \in M$ such that $\mathcal{SM}(M) = \bigcup_{i=1}^n S(a_i)= S(a_0)$, where $a_0=a_1\vee \cdots \vee a_n$. Let $I_{a_0}$ be the ideal of $M$ generated by $a_0$. If we set $S(I_{a_0})=\{s \in \mathcal{SM}(M) \colon \Ker(s) \not\supseteq I_{a_0}\}$, then $S(I_{a_0})=S(a_0)$. We assert that $I_{a_0} = M$, if not then $I_{a_0}$ is a proper ideal of $M$, and there is a maximal ideal $I$ of $M$ containing $I_{a_0}$. Due to Theorem \ref{th:state6}, $I=\Ker(s)$ for some $s \in \mathcal{SM}(M)$, which implies $s\in S(a_0)$ and $s\notin S(I_{a_0})=S(a_0)$, a contradiction. Therefore, $I_{a_0}=M$ which means that for each $x \in M$, $x\in I_{a_0}$ and consequently, $x \le n.a_0=a_0$, confirming $a_0$ is a top element of $M$.
\end{proof}

If $s_1,\ldots,s_n \in \mathcal S(M)$ and  real numbers $\lambda_1,\ldots,\lambda_n \in [0,1]$ satisfy $\sum_{i=1}^n \lambda_i =1$, then $s= \sum_{i=1}^n \lambda_is_i$ is also a state of $M$ and $s$ is said to be a {\it convex combination} of $s_1,\ldots,s_n$. If $X$ is a non-empty set of states, then $\Con(X)$ means the {\it convex hull} generated by $X$, i.e. $\Con(X)$ is the set of all convex combinations of states from $X$. We denote by $(\Con(X))^-$ the closure of $\Con(X)$ in the weak topology of states. If $M$ has a top element, i.e. $M$ is in fact an MV-algebra, then due to Krein--Mil'man theorem, see \cite[Thm 5.17]{Goo}, $\mathcal S(M) = (\Con(\partial \mathcal S(M)))^-$. Since the Krein--Mil'man theorem is formulated for compact convex sets, if $M$ has no top element, as we have seen in Proposition \ref{pr:state9}, $\mathcal S(M)$ is not compact, so that we cannot apply directly the Krein--Mil'man theorem for $\mathcal S(M)$. In what follows, Theorem \ref{th:state12} below, we show that anyway we have

\begin{equation}\label{eq:KM}
\mathcal S(M)= (\Con(\mathcal{SM}(M)))^{-_M},
\end{equation}
where $^{-_M}$ denotes the closure taken in the weak topology of states on $M$.

To prove that, we use the Basic Representation Theorem, see \cite[Thm 5.21]{DvZa} or Theorem \ref{th:embed}, which says that for any EMV-algebra, there is an MV-algebra $N$ such that either $M=N$ (if $M$ has a top element) or $M$ is a maximal ideal of $N$ (if $M$ has no top element), and each element $x$ of $N$ is either $x=x_0\in M$ or $x=\lambda_1(x_0)$ for some element $x_0 \in M$, where $1$ is the top element of $N$. States and state-morphisms on $N$ we can describe as follows.

\begin{proposition}\label{pr:state11}
Let $M$ be an EMV-algebra without top element. For each $x \in M$, we put $x^*=\lambda_1(x)$, where $1$ is the top element of the representing MV-algebra $N$.
Given a state $s$ on $M$, the mapping $\tilde s: N \to [0,1]$, defined by
\begin{equation}\label{eq:state}
\tilde s(x)=\begin{cases}
s(x) & \text{ if } x\in M,\\
1-s(x_0) & \text{ if }  x=x^*_0,\ x_0 \in M,
\end{cases} \quad x \in N,
\end{equation}
is a state on $N$, and the mapping $s_\infty: N \to [0,1]$ defined by $s_\infty(x)=0$ if $x\in M$ and $s_\infty(x)=1$ if $x\in N\setminus M$, is a state-morphism on $N$. If $s$ is a state-morphism on $M$, then $\tilde s$ is a state-morphism on $N$.
Moreover, $\mathcal{SM}(N)=\{\tilde s \mid s \in \mathcal{SM}(M)\}\cup\{s_\infty\}$ and $\Ker(\tilde s) = \Ker(s) \cup \Ker_1^*(s)$, $s \in \mathcal{SM}(M)$, where $\Ker^*_1(s)=\{\lambda_1(x)\mid x \in \Ker_1(s)\}$.

A net $\{s_\alpha\}_\alpha$ of states on $M$ converges weakly to a state  $s$ on $M$ if and only if $\{\tilde s_\alpha\}_\alpha$ converges weakly to $\tilde s$ on $N$, and the mapping $\phi: \mathcal S(M)\to \mathcal S(N)$ defined by $\phi(s)=\tilde s$, $s \in \mathcal S(M)$, is injective, continuous and affine.
\end{proposition}

\begin{proof}
Due to Mundici's result, there is an Abelian unital $\ell$-group $(G,u)$ such that $N \cong \Gamma(G,u)$. Without loss of generality, we can assume $N=\Gamma(G,u)$. Then if $x=x^*_0$ for some $x_0 \in M$, then $x= u-x_0$, where $-$ is the group subtraction taken from the group $G$.

Let $s$ be a state on $M$ and define $\tilde s$ by (\ref{eq:state}). Then $\tilde s(1)=1$. Let $x,y \in N$ and $x\odot y = 0$. There are three cases: (i) $x=x_0,y=y_0\in M$. Then $\tilde s(x+y)= s(x_0+y_0)=s(x_0)+s(y_0)= \tilde s(x)+\tilde s(y)$. (ii) $x= x_0\in M$ and $y=y^*_0$ where $y_0 \in M$. $x\odot y = 0$ implies $u-y_0\le u- x_0$, i.e. $x_0\le y_0$. There is an idempotent $a\in \mathcal I(M)$ such that $x_0\le y_0 \le a$ and $s(a)=1$.
Since $x+y=x\oplus y =x_0\oplus y_0^*= (y_0\odot x^*_0)^*= (y_0\odot \lambda_a(x_0))^*$, see (2.2), it yields by Proposition \ref{pr:state3}(ii) $\tilde s(x\oplus y)=1-s(y_0\odot \lambda_a(x_0))= 1- s(y_0)+s(x_0)=\tilde s(x)+\tilde s(y)$. (iii) $x=x^*_0$, $y=y^*_0$ for some $x_0,y_0 \in M$. Then $x\odot y = 0$ entails $x^*_0\le y_0$, i.e. $u-y_0 \le x_0$ and $u \le x_0 + y_0 \in M$, so that $u \in M$ which is absurd. So that this case is impossible. Therefore, $\tilde s$
is a state on $N$.

If $s$ is a state-morphism, we proceed in a similar way as for states. Let $x,y \in N$. We have again three cases: (i) $x=x_0$, $y=y_0$, $x_0,y_0 \in M$, which is trivial. (ii) $x=x_0$, $y=y^*_0$ for $x_0,y_0\in M$. Then there exists an idempotent $a\in \mathcal I(M)$ such that $x_0,y_0 \le b$ and $s(a)=1$. Since $x\oplus y =x_0\oplus y_0^*= (y_0\odot x^*_0)^*= (y_0\odot \lambda_a(x_0))^*$ which yields $\tilde s(x\oplus y)=1-s(y_0\odot \lambda_a(x_0))= 1-(s(y_0)\odot (s(a)-s(x_0)) =(1-s(y_0))\oplus s(x_0) = \tilde s(x)\oplus s(y)$.
(iii) Let $x=x^*_0$, $y=y^*_0$ for $x_0,y_0\in M$. Then $x\oplus y = (x_0\odot y_0)^*$, so that $\tilde s(x\oplus y)= 1- s(x_0\odot y_0)=1-s(x_0)\odot s(y_0)= \tilde s(x)\oplus \tilde s(y)$.

The mapping $s_\infty$ is evidently a state-morphism on $N$. Now let $s$ be any state-morphism on $N$. There are two cases: (i) For each idempotent $a \in M$, we have $s(a)=0$. Then $s(x)=0$ for each $x \in M$, i.e. $s = s_\infty$. (ii) There is an idempotent $a\in M$ such that $s(a)=1$. Then the restriction of $s$ onto $M$ is a state-morphism on $M$, say $s_0$, so that that $s = \tilde s_0$.

The rest properties are straightforward.
\end{proof}

According to \cite[Thm 4.10]{DvZa1}, if $M$ has no top element, then $\mathcal{SM}(M)$ is locally compact but not compact. For the state space $\mathcal S(M)$, we have it is even not locally compact as it follows from the following result.

\begin{theorem}\label{th:local}
Let $M$ be an EMV-algebra. The following statements are equivalent:
\begin{itemize}
\item[{\rm (i)}]  The state space $\mathcal S(M)$ is locally compact.
\item[{\rm (ii)}] The state space $S(M)$ is compact.
\item[{\rm (iii)}] $M$ has a top element.
\end{itemize}
\end{theorem}

\begin{proof}
If $M$ has a top element, then $\mathcal S(M)$ is compact, so that it is locally compact, i.e. (iii) $\Rightarrow$ (ii) $\Rightarrow$  (i), and by Proposition \ref{pr:state9}, (ii) and (iii) are equivalent.

(i) $\Rightarrow$  (ii) Now, assume that $X=\mathcal{S}(M)$ is locally compact but not compact in the weak topology of states, therefore, $M$ has no top element. Let $N$ be the MV-algebra representing $M$ such that $M$ is a maximal ideal of $N$, and every element $x \in N$ either belongs to $M$ or $\lambda_1(x)\in M$. According to the Alexander theorem, see  \cite[Thm 4.21]{Kel}, there is a compact space $X^*=X \cup \{x_\infty\}$, where $x_\infty \notin X$. Define a mapping $\phi:\mathcal S(M) \to \mathcal S(N)$ given by $\phi(s)=\tilde s$, $s \in \mathcal S(M)$, where $\tilde s$ is defined by (\ref{eq:state}). Then a net of states $\{s_\alpha\}_\alpha$ converges weakly to a state $s\in \mathcal S(M)$ iff $\{\tilde s_\alpha\}_\alpha$ converges weakly to $\tilde s$ on $N$. Therefore, $\phi$ maps $X$ onto the set $\phi(X)=\{\tilde s \colon s \in \mathcal S(M)\}$, so that $\phi$ is a homeomorphism from $X$ onto $\phi(X)$. Then also $\phi(X)$ has the one-point compactification $(\phi(X))^*=\phi(X)\cup\{x^*_\infty\}$, where $x^*_\infty \notin \phi(X)$. But for the state-morphism $s_\infty$ on $N$ given by $s_\infty(x)=0$ if $x \in M$ and $s_\infty(x)=1$ for $x \in N\setminus M$, there is a net $\{t_\beta\}_\beta$ of state-morphisms on $M$, such that $\{\tilde t_\beta\}_\beta$ converges weakly to $s_\infty$ on $N$, for more details see \cite[Thm 4.13]{DvZa1}. Therefore, $t_\beta \in X$ and $\tilde t_\beta \in \phi(X)$ for each index $\beta$ and $\lim_\beta t_\beta(x)=0$ for each $x \in M$ and $\tilde s_0=s_\infty$. On the other hand, since $(\phi(X))^*$ is compact, there is a subnet $\{\tilde t_{\beta_\alpha}\}_\alpha$ of the net $\{\tilde t_\beta\}_\beta$ which converges to some point $x^*\in \phi(X)\cup \{x^*_\infty\}$. Then $x^*=x^*_\infty=s_\infty$.

Now let $s$ be any state-morphism on $M$ and for each $\lambda \in (0,1)$ we set $t^\lambda_\beta = \lambda s+ (1-\lambda)t_\beta$. Then $t^\lambda_\beta\in X$ and $\phi(t^\lambda_\beta)= \lambda \tilde s +(1-\lambda)\tilde t_\beta \in \phi(X)$ for each index $\beta$. Since $\{\phi(t^\lambda_\beta)\}_\beta$ converges weakly on $N$ to $\lambda \tilde s +(1-\lambda)s_\infty$ so that $\lambda \tilde s +(1-\lambda)s_\infty \in (\phi(X))^*=\phi(X)\cup \{s_\infty\}$. But $\lambda \tilde s +(1-\lambda)s_\infty$ gives for each $\lambda \in (0,1)$ countably many mutually different states on $N$ not belonging to $\phi(X)$, which says that there is no one-point compactification of $\mathcal S(M)$. Hence, our assumption that $\mathcal S(M)$ is not compact was wrong, and $\mathcal S(M)$ has to be compact.
\end{proof}

Now we establish (\ref{eq:KM}) for each EMV-algebra.

\begin{theorem}\label{th:state12} {\rm [Krein--Mil'man Representation of States]}
Let $M$ be an EMV-algebra. Then
$$\mathcal S(M)= (\Con(\mathcal{SM}(M)))^{-_M}$$
where $^{-_M}$ denotes the closure in the weak topology of states on $M$.
\end{theorem}

\begin{proof}
If $M$ has a top element, then $\mathcal S(M)$ is a compact set in the weak topology. A direct application of the Krein--Mil'man theorem to Theorem \ref{th:state8} gives the result. If $M= \{0\}$, then $\mathcal S(M)=\mathcal{SM}(M)=\emptyset$, so that the result holds also in this case. Finally, let $M\ne\{0\}$.

Now, let $M$ have no top element. Using the Basic Representation Theorem for EMV-algebras, see Theorem \ref{th:embed}, there is an MV-algebra $N$ such that $M$ is its maximal ideal and every element $x\in N$ is either $x \in M$ or $x^* \in M$. Then for the state space of $N$ we have $\mathcal S(N)= (\Con(\mathcal{SM}(N)))^{-_N}$, where $^{-_N}$ is the closure in the weak topology of states on $N$. Take an arbitrary state $s$ on $M$ that is not extremal, equivalently, $s$ is not a state-morphism on $M$. There is a net $\{s_\alpha\}_\alpha$ of convex combinations from $\mathcal{SM}(N)$ such that $\{s_\alpha\}_\alpha$ converges weakly to $\tilde s$ on $N$.  Since also $\tilde s$ is not an extremal state on $N$, without loss of generality we can assume that each $s_\alpha$ is not a state-morphism.

In addition, let $s_\alpha = \lambda_0^\alpha s_\infty + \sum_{i=1}^{n_\alpha} \lambda^\alpha_i \tilde s^\alpha_i$, where all $\lambda$'s are from $[0,1]$, $\sum_{i=0}^{n_\alpha} \lambda^\alpha_i =1$, and $s^\alpha_i \in \mathcal{SM}(M)$ for $i=1,\ldots,n_\alpha$ and for each $\alpha$. If there is an index $\alpha_0$ such that for each $\alpha>\alpha_0$, we have $\lambda^\alpha_0 =0$ which gives $s \in (\Con(\mathcal{SM}(M)))^{-_M}$. Therefore, we can assume also that each $\lambda^\alpha_0>0$, or to pass to its subnet with such a property, if necessary. In addition, we can assume  $\lambda_0^\alpha <1$ for each $\alpha$, otherwise $s_\alpha =s_\infty$, $s_\alpha(x)=0$ and $s(x)=0$ for each $x \in M$, which is impossible.

Since $s$ is a state on $M$, there is an element $a \in M$ such that $s(a)=1$. Then $s_\alpha(a)= \sum_{i=1}^{n_\alpha} \lambda^\alpha_i \tilde s^\alpha_i(a) \le \sum_{i=1}^{n_\alpha} \lambda^\alpha_i = 1-\lambda^\alpha_0\le 1$. Then $1= \liminf_\alpha s_\alpha(a) \le \liminf_\alpha (1-\lambda^\alpha_0)\le 1$, so that $\limsup_
\alpha \lambda^\alpha_0=1$ and similarly, $1= \limsup_\alpha s_\alpha(a) \le \limsup_\alpha (1-\lambda^\alpha_0)\le 1$, i.e. $\liminf_\alpha \lambda^\alpha_0 = 1$. Whence, $\lim_\alpha \lambda^\alpha_0$ exists and $\lim_\alpha \lambda^\alpha_0=1$. In addition, $t_\alpha:=s_\alpha/(1-\lambda^\alpha_0) = \sum_{i=1}^{n_\alpha} \lambda^\alpha_i /(1-\lambda^\alpha_0)s^\alpha_i$ and $t_\alpha \in \Con(\mathcal{SM}(M))$. Moreover, the net $\{t_\alpha\}_\alpha$ converges weakly to $s$ on $M$, so that $s \in (\Con(\mathcal{SM}(M)))^{-_M}$ and finally, $\mathcal S(M)= (\Con(\mathcal{SM}(M)))^{-_M}$.
\end{proof}

\begin{proposition}\label{pr:fin}
Let $M$ be an EMV-algebra. If $\mathcal S(M)=\mathcal{SM}(M)$, then $\mathcal S(M)$ and $\mathcal{SM}(M)$ both are either empty sets or singletons and, in addition, $M$ is an MV-algebra.
\end{proposition}

\begin{proof}
If $\mathcal S(M)$ is empty, then $M=\{0\}$ and $M$ is a degenerate (= one-element) MV-algebra. Now let $\mathcal S(M)\ne \emptyset$, then $M\ne \{0\}$. By Theorem \ref{th:state8}, $\partial \mathcal S(M) = \mathcal{SM}(M)$, so that both sets $\mathcal S(M)$ and $\mathcal{SM}(M)$ are singletons. Due to Theorem \ref{th:state6}, $M$ has a unique maximal ideal which by \cite[Thm 3.25]{DvZa} means that $M$ is an MV-algebra.
\end{proof}

Now we describe those states on $N$ whose restrictions to $M$ are not states on $M$.

\begin{proposition}\label{pr:state13}
Let $M$ be an EMV-algebra without top element. Then the restriction of a state $s \in \mathcal S(N)$ to $M$ is a state on $M$ if and only if $s \in \phi(\mathcal S(M))$, where $\phi:\mathcal S(M)\to \mathcal S(N)$ is given by $\phi(s)=\tilde s$, which is defined by {\rm (\ref{eq:state})}.

In particular, there is uncountable many states on $N$ whose restriction to $M$ is not a state on $M$.
\end{proposition}

\begin{proof}
If $s \in \phi(\mathcal S(M))$, then there is a state $s_0$ on $M$ such that $s=\tilde s_0$. Therefore, $s|_M=s_0$ is a state on $M$. Now let $s_0:=s|_M \in \mathcal S(M)$, then clearly $\phi(s_0)=s$.

Let $s$ be a state-morphism on $N$ different of $s_\infty$ onto $M$, $\lambda \in (0,1)$, and let $s_\lambda = \lambda s +(1-\lambda)s_\infty$. Then for the restriction of $s_\lambda$ onto $M$, we have $s_\lambda|_M=\lambda s|_M$ which is not a state on $M$, and the system $\{s_\lambda \colon \lambda \in (0,1)\}$ gives an uncountable system of mutually different states on $M$ whose restriction to $M$ is not a state on $M$.
\end{proof}

Given an element $x \in M$, we set $0x=0$, $1x= x$ and $nx=(n-1)x +x$, $n \ge 2$, if $(n-1)x$ and $(n-1)x+x$ are defined in $M$.
We say that an element $x \in M$ is said to be an {\it infinitesimal} if the element $nx$ is defined in $M$ for each integer $n\ge 1$. We denote by $\Infinit(M)$ the set of all infinitesimal elements of $M$. 

\begin{proposition}\label{pr:state14}
Let $M$ be an EMV-algebra. Then $\Infinit(M)$ is a proper ideal of $M$ and
$$\Infinit(M)=\Rad(M) =\{x\in M \colon s(x) =0 \text{ for each } s\in \mathcal{SM}(M)\}.
$$
\end{proposition}

\begin{proof}
By Theorem \ref{th:state6}, $\Rad(M)=\{x\in M \colon s(x) =0 \text{ for each } s\in \mathcal{SM}(M)\}$.
Let $x \in \Infinit(M)$, then $s(nx)=ns(x)\le 1$ and $s(x)\le 1/n$ for each $n\ge 1$, so that $s(x)=0$ for each state-morphism $s$ on $M$, i.e. $\Infinit(M) \subseteq \Rad(M)$.

Now let $x>0$ be not infinitesimal, and let $x \le a\in \mathcal I(M)$. There is an integer $n$ such that $nx\not\le \lambda_a(x)$. Using (\ref{eq:x<y}), we have $c:= (nx) \odot \lambda_a((nx)\wedge \lambda_a(x))= (nx)\odot \lambda_a(\lambda_a(x))= (nx)\odot x >0$. There is an ideal $P$ which is maximal under the condition $c\notin P$. By \cite[Thm 5.12]{DvZa}, $P$ is prime and is contained in a unique maximal ideal $I$ of $M$, see \cite[Prop 5.9]{DvZa}. There is a unique state-morphism $s$ on $M$ such that $I=\Ker(s)$, in addition, there is an idempotent $b\in \mathcal I(M)$ such that $x\le a \le b$ and $s(b)=1$. Then $nx\not\le \lambda_b(x)$, otherwise, $(nx)\odot x=0$. Therefore, $c=(nx)\odot x = (nx) \odot \lambda_b(nx\wedge \lambda_b(x)) = (nx)\odot \lambda_b(\lambda_b(x))$. Then $\big((nx) \odot \lambda_b(\lambda_b(x))\big)\wedge \big(\lambda_b(x)\odot \lambda_b(nx)\big)=0$ which implies  $\lambda_b(x)\odot \lambda_b(nx) = \lambda_b(x\oplus nx)=\lambda_b((n+1).x) \in P\subseteq I$. Then $s(\lambda_b((n+1).x))= s(b)- s((n+1).x)=0$ and $1=s(b)= s((n+1).x)=(n+1).s(x)$ which entails $s(x)>0$, $x \notin I$, and $x \notin \Rad(M)$. Therefore, $\Rad(M) \subseteq \Infinit(M)$, which gives the result $\Infinit(M)=\Rad(M)$. Since clearly $\Rad(M)$ is an ideal so is $\Infinit(M)$ and $\Infinit(M)$ is a proper ideal of $M$.
\end{proof}

\begin{proposition}\label{pr:state15}  Let $M$ be an EMV-algebra. The state spaces of $M$ and $M/\Rad(M)$ are affinely homeomorphic.
\end{proposition}

\begin{proof}
For any $x \in M$, let $[x]:=x/\Rad(M)$. Let $s$ be a state on $M$.
Then the mapping $\hat s$ on $M/\Rad(M)$ defined by $\hat
s([x])=s(x)$ ($x\in M$) is a state on $M/\Rad(M)$. Indeed, assume $[x]=[y]$,
then $s'(x)=s'(x\wedge y)=s'(y)$ for any state-morphism $s'$ on $M$ so
that by Theorem \ref{th:state12}, $s(x)=s(x\wedge y)=s(y)$ for any state $s$ on $M$. Then $\hat s$ is well defined, and $\hat s([x])=1$ whenever $s(x)=1$.  Assume $[x]+ [y]$ is defined in $M/\Rad(M)$, and let $x,y \le a\in \mathcal I(M)$ with $s(a)=1$. To prove $\hat s([x]+[y])=\hat s([x])+\hat s([y])$, we use the same steps as those in the proof of (vii) of Proposition \ref{pr:state3}. Consequently, $\hat s$ is a state on $M/\Ker(M)$.

From the characterization of extremal states on EMV-algebras, Theorem \ref{th:state6} and Theorem \ref{th:state8}, we see that if  $s$ is extremal so is $\hat s$.

Conversely, if $\mu$ is a state on $M/\Rad(M)$, then the mapping
$s_\mu(x)=\mu([x])$ for $x\in M$ is a state on $M$, and if $\mu$ is extremal so is $s_\mu$.  Moreover, $\widehat{s_\mu}=\mu$.

The mapping $s\mapsto \hat s$ is therefore injective, surjective, continuous, open, and affine.
\end{proof}

\section{Pre-states}

Besides states on EMV-algebras we define pre-states, strong pre-states, and pre-state-morphisms. They are of a weaker form than states state-morphisms and they are important mainly when an EMV-algebra has no top element, but in such a case, they can be extended to states on the representing MV-algebra.

Let $M$ be an EMV-algebra. We say that a mapping $s:M \to[0,1]$ is (i) a {\it pre-state} if $s(x+y)=s(x)+s(y)$ whenever $x+y$ is defined in $M$, and (ii) a {\it pre-state-morphism} if $s(x\oplus y)=s(x)\oplus s(y)$, $x,y \in M$, where $u\oplus v:=\min\{u+v,1\}$ for all $u,v \in [0,1]$. Properties: (i) $s(0)=0$, and (ii) if $x\le y$, from $y=x+(y\odot \lambda_a(x))=x\oplus (y\odot \lambda_a(x))$, where $x,y\le a \in \mathcal I(M)$, we conclude that $s(x)\le s(y)$ and $s(y\odot \lambda_a(x))=s(y)-s(x)$.

We denote by $\mathcal{PS}(M)$ and $\mathcal{PSM}(M)$ the set of pre-states and pre-state-morphisms, respectively, on $M$. For example, if $M$ has no top element and $N$ is its representing MV-algebra, then the restriction of any state on $N$ onto $M$ is a pre-state on $M$. The restriction of $s_\infty$ onto $M$ is the zero function on $M$. It is clear, that the set $\mathcal{PS}(M)$ is a convex set; we note that extremal pre-states are defined in the same way as do extremal states and they are described in Theorem \ref{th:int2} below.

A pre-state $s$ on $M$ is said to be a {\it strong pre-state} if there is an element $x_0\in M$ such that $s(x_0)=\sup\{s(x)\colon x \in M\}$. Then every state on $M$ is a strong pre-state and if $M$ is a $\sigma$-complete EMV-algebra, i.e. every sequence of elements in $M$ has a supremum, then every pre-state on $M$ is strong. Indeed, let $r=\sup\{s(x) \colon x \in M\}$. There is a sequence $\{x_n\}$ of elements of $M$ such that $x_1\le x_2\le \cdots$ and $r=\lim_n s(x_n)$. Put $x_0=\bigvee_n x_n$, then $r\ge s(x_0)\ge s(x_n)$ so that $r=s(x_0)$. If $M$ has a top element, then every pre-state is strong, and if $s_0$ is a pre-state, there is a state $s$ on $M$ and a number $r \in [0,1]$ such that $s_0=rs$. If $s_0$ is non-zero, there is a unique state $s$ and unique number $r\in (0,1]$ such that $s_0=rs$.

Below, see Theorem \ref{th:int2}, it will be proved that every pre-state-morphism is a strong pre-state, more precisely, we show that every pre-state-morphism on $M$ is either the zero function or a state-morphism.

We denote by $\mathcal{PS}_s(M)$ the set of strong pre-states on $M$. The sets $\mathcal{PS}(M)$ and $\mathcal{PS}_s(M)$ are convex sets containing $\mathcal{PSM}(M)$. The restriction of any convex combination of state-morphisms on $N$ onto $M$ is a state on $M$. On the other hand if $s$ is a state-morphism on $N$ different of $s_\infty$, then the restriction of $s_\lambda:=\lambda s + (1-\lambda) s_\infty$ onto $M$, where $\lambda \in (0,1)$, is a strong pre-state on $M$ such that the maximal value of the restriction of $s_\lambda$ onto $M$ is $\lambda$. This follows from the fact, see \cite[Prop 4.4]{DvZa1}, that $\mathcal{SM}(N)=\phi(\mathcal{SM}(M))\cup \{s_\infty\}$, where $\phi$ was defined in Proposition \ref{pr:state11}.

We note that the restriction of any state on $N$ onto $M$ is not necessarily a strong pre-state on $M$ as it follows from the following example.

\begin{example}\label{ex:contra}
Let $\mathcal T$ be an $EMV$-algebra from Example {\rm \ref{ex:state9}}. Then $\mathcal T$ is an EMV-algebra without top element and $\mathcal T$ is not $\sigma$-complete. Its representing MV-algebra is $\mathcal N:=\{A \subseteq \mathbb N \colon \text{ either } A \text{ is finite or } \mathbb N\setminus A \text{ is finite}\}$. Then $\mathcal{SM}(\mathcal T)= \{s_n\colon n \in \mathbb N\}$, where $s_n(A)=\chi_A(n)$, $A \in \mathcal T$, and $\mathcal{SM}(\mathcal N)=\{\tilde s_n \colon n \in \mathbb N\}\cup\{s_\infty\}$. Take a state $s= \lambda_0 s_\infty +\sum_{n=1}^\infty \lambda_n\tilde s_n$, where each $\lambda$ is from $(0,1)$ and $\sum_{n=0}^\infty \lambda_n =1$. Then the restriction of $s$ onto $\mathcal T$ is the function $s_0=\sum_{n=1}^\infty \lambda_n s_n$ which is a pre-state but not a strong pre-state on $\mathcal T$ because if  $A_k=\{1,\ldots,k\}$ for each $k\ge 1$, we have, given $A \in \mathcal T$, there is an integer $k\ge 1$ such that $A \subseteq A_k$ which gives $s_0(A)\le s_0(A_k)= \sum_{i=1}^k\lambda_i<\sum_{i=1}^\infty \lambda_i = \sup\{s_0(A)\colon A \in \mathcal T\}$.
\end{example}

Further properties of states and pre-states on $\mathcal T$ are presented in Examples \ref{ex:integ3}--\ref{ex:integ4}, Theorem \ref{th:integ5}, and Corollary \ref{co:integ6}.

If $s$ is a state on $M$ and $r$ is a real number $r\in[0,1)$, then $s_r(x) =rs(x)$, $x \in M$, is a strong pre-state which is not a state. In particular, the zero function $s_0$ on $M$ is both a pre-state and a pre-state-morphism as well. Clearly, if $s$ is a strong pre-state not vanishing on $M$ and $r=\sup\{s(x)\colon x \in M\}$, then $s_r:=\frac{1}{r}s$ is a state on $M$.

For pre-states on an EMV-algebra, we define the weak topology in a standard way. Then $\mathcal{PS}(M)$ is a compact set in the weak topology.


\begin{proposition}\label{pr:int0}
Let $M$ be an EMV-algebra without top element and $N$ be its representing MV-algebra. Every pre-state on $M$ is a restriction of a unique state on $N$. That is, if $s$ is a pre-state on $M$, the mapping $\tilde s: N\to [0,1]$ defined by
\begin{equation}\label{eq:state0}
\tilde s(x)=\begin{cases}
s(x) & \text{ if } x\in M,\\
1-s(x_0) & \text{ if }  x=x^*_0,\ x_0 \in M,
\end{cases} \quad x \in N,
\end{equation}
is a unique state on $N$ whose restriction to $M$ coincides with $s$. Moreover, the spaces $\mathcal{PS}(M)$ and $\mathcal S(N)$ are affinely homeomorphic in the weak topologies.
\end{proposition}

\begin{proof}
Let $s$ be a pre-state on $M$. Let us define a mapping $\tilde s: N\to [0,1]$ by (\ref{eq:state0}).  We assert $\tilde s$ is a state on $N$ whose restriction to $M$ is $s$. Indeed, $\tilde s(1)=1$, and similarly as in the proof of Proposition \ref{pr:state11},  for $x,y \in N$ with $x\odot y=0$, we have three cases: (i) $x=x_0,y=y_0\in M$, (ii) $x= x_0\in M$ and $y=y^*_0$ where $y_0 \in M$, and (iii) $x=x^*_0$, $y=y^*_0$ for some $x_0,y_0 \in M$. The first case is straightforward, and the third one is impossible. Check for (ii) $x\odot y = 0$ implies $u-y_0\le u- x_0$, i.e. $x_0\le y_0$. There is an idempotent $a\in \mathcal I(M)$ such that $x_0\le y_0 \le a$.
Since $x+y=x\oplus y =x_0\oplus y_0^*= (y_0\odot x^*_0)^*= (y_0\odot \lambda_a(x_0))^*$ which yields $\tilde s(x\oplus y)=1-s(y_0\odot \lambda_a(x_0))= 1- s(y_0)+s(x_0)=\tilde s(x)+\tilde s(y)$.

Therefore, $\tilde s$ is a state on $N$ whose restriction to $M$ coincides with $s$. We note that if $s$ is the zero function, then $\tilde s= s_\infty$. From (\ref{eq:state0}) we conclude that $\tilde s$ is a unique state on $N$ whose restriction to $M$ is $s$.

The mapping $\kappa: \mathcal{PS}(M)\to \mathcal S(N)$ given by $\kappa(s)=\tilde s$, $s \in \mathcal{PS}(M)$, is continuous, affine, and since the restriction $\hat s$ of a state $s$ on $N$ to $M$ is a pre-state, we have $\tilde{\hat s}=s$, so that $\kappa$ is invertible. Since a net $\{s_\alpha\}_\alpha$ of pre-states on $M$ converges weakly to a pre-state $s$ on $M$ iff $\{\tilde s_\alpha\}_\alpha$ converges weakly to $\tilde s$ on $N$, we see that $\kappa$ is an affine homeomorphism.
\end{proof}

\begin{proposition}\label{pr:int1}
If $M$ is an EMV-algebra which has no top element, there is another way of extension of strong pre-states on $M$.
\end{proposition}

\begin{proof}
If $s$ is the zero function on $M$, then $s$ is the restriction of the state-morphism $s_\infty$, which was defined in Proposition \ref{pr:state11}. Now let $s$ be a non-vanishing strong pre state on $M$, and let $r=\sup\{s(x)\colon x \in M\}$. Then $r \in (0,1]$ and $s_r:=\frac{1}{r}s$ is a state on $M$, so by Proposition \ref{pr:state11}, there is a unique state $\tilde s_r$ on $N$ defined by (\ref{eq:state}). If we define a mapping $\tilde s$, a convex combination,
\begin{equation}\label{eq:state1}
\tilde s=r\tilde s_r +(1-r)s_\infty,
\end{equation}
then $\tilde s$ is a state on $N$ such that its restriction to $M$ is $s$. It is interesting to note that formula (\ref{eq:state1}) works also for the zero function $s_0$ on $M$; then $\tilde s_0=s_\infty$.
We note that if $s$ is a state on $M$, then $\tilde s$ coincides with the state on $N$ defined by (\ref{eq:state}). Let $\kappa$ be a mapping from $\mathcal{PS}_s(M)$ into $\mathcal S(N)$ defined by $\kappa(s)=\tilde s$, $s \in \mathcal{PS}_s(M)$, where $\tilde s$ is defined by (\ref{eq:state1}). Then $\kappa$ is injective. Indeed, let $\kappa(s_1)=\kappa (s_2)$ for $s_1,s_2 \in \mathcal{SM}_s(M)$. Let $r_i=\max\{s_i(x) \colon x \in M\}$ for $i=1,2$. Then there is an element $x\in M$ such that $s_i(x)=r_i$ so that $r_1=r_2$ which yields $s_1=s_2$.

Now let $s_1,s_2$ be two strong pre-states and let $r_i=\max\{s_i(x)\colon x\in M\}$ for $i=1,2$. There is an element $x \in M$ such that $r_i=s_i(x)$.  Then for the convex combination $s= \lambda s_1 +(1-\lambda)s_2$, where $\lambda \in [0,1]$, we have that $s$ is also a strong pre-state such that for $r=\sup\{s(x) \colon x \in M\}$ we have $r = \lambda r_1 +(1-\lambda) r_2$, $\mathcal{PS}_s(M)$ is a convex set, and $\kappa(s)=\lambda \kappa(s_1)+(1-\lambda)\kappa(s_2)$.
\end{proof}

\begin{theorem}\label{th:int2}
Let $M$ be an EMV-algebra. Then any pre-state-morphism is either a state-morphism or the zero function $s_0$ on $M$, that is, $\mathcal{PSM}(M) = \mathcal{SM}(M) \cup \{s_0\}$, and every pre-state-morphism is a strong pre-state. The space $\mathcal{PSM}(M)$ is compact and homeomorphic to $\mathcal{SM}(N)$, where $N$ is the representing MV-algebra for $M$. The set $\mathcal{PS}(M)$ is compact and the closure of $\mathcal{SM}(M)$ in the weak topology of pre-states is equal to $\mathcal{PSM}(M)$.

The set $\partial \mathcal{PS}(M)$ of extremal pre-states on $M$ is the set of pre-state-morphisms, i.e. $\partial \mathcal{PS}(M) = \mathcal{PSM}(M)=\mathcal{SM}(M)\cup \{s_0\}$. Moreover,
\begin{equation}\label{eq:KMp}
\mathcal{PS}(M)=(\Con(\mathcal{PSM}(M)))^{-_M}.
\end{equation}
\end{theorem}

\begin{proof}
If $\{s_\alpha\}_\alpha$ is a net of pre-states on $M$ and $s(x)=\lim_\alpha s_\alpha(x)$ for each $x \in M$, then $s$ preserves partial addition $+$, so that $\mathcal{PS}(M)$ is a compact set in the product topology. Similarly, the set $\mathcal{PSM}(M)$ is compact.

Let $s$ be a pre-state-morphism on $M$. There are two cases: (i) There is an idempotent $a$ such that $s(a)=1$, then $s$ is a state-morphism. (ii) For each idempotent $a\in M$ we have $s(a)<1$. In view of $s(a)= s(a\oplus a) = \min\{s(a)+s(a),1\}= s(a)\oplus s(a)$,  we have $s(a)=0$ because the MV-algebra of the real interval $[0,1]$ has only two idempotents, namely $0$ and $1$. Since every pre-state-morphism is monotone, we have $s(x)=0$ for each $x\in M$, i.e. $s$ is the zero function $s_0$ which proves $\mathcal{PSM}(M) = \mathcal{SM}(M) \cup \{s_0\}$. By Proposition \ref{pr:state4}, each pre-state morphism is a strong pre-state.

Now let $\{s_\alpha\}_\alpha$ be a net of pre-state-morphisms on $M$ and let $s(x)=\lim_\alpha s_\alpha(x)$ exists for each $x \in M$. Then $s$ is clearly a pre-state-morphism because it preserves $\oplus$. Due to the preceding paragraph, we have that the closure of $\mathcal{SM}(M)$ is the set of pre-state-morphisms on $M$.

If $M$ has a top element, then $M=N$ and clearly $\mathcal S(M) =\mathcal S(N)$ and $\mathcal{SM}(M) = \mathcal{SM}(N)$. Now let $M$ have no top element.

We note that according to \cite[Thm 4.13]{DvZa1}, the set of state-morphisms $\mathcal{SM}(N)$ is the one-point compactification of $\mathcal{SM}(M)$ and the mapping $s\mapsto \tilde s$, $s \in \mathcal{SM}(M)$, where $\tilde s$ is defined by (\ref{eq:state}), is a continuous embedding, and the zero pre-state $s_0$ maps to $s_\infty$. Whence, we have that $\mathcal{PSM}(M)$ is homeomorphic to $\mathcal{SM}(N)$.

By Proposition \ref{pr:int0}, the mapping $\kappa:\mathcal{PS}(M)\to \mathcal{S}(N)$, given by $\kappa(s)=\tilde s$, $s\in \mathcal{SM}(M)$, and by (\ref{th:int2}), is an affine homeomorphism, and by Theorem \ref{th:state8}, $\partial \mathcal S(N)=\mathcal{SM}(N)=\{\tilde s\colon s\in \mathcal{SM}(M)\}\cup\{s_\infty\}$ which gives the result $\partial \mathcal{PS}(M)=\mathcal{SM}(M)\cup\{s_0\}$, where $s_0$ is the zero function on $M$.

Since $\mathcal{PS}(M)$ is convex and compact in the weak topology of pre-states, applying the Krein--Mil'man theorem \cite[Thm 5.17]{Goo}, we obtain immediately (\ref{eq:KMp}).
\end{proof}

As it was already said, in \cite[Thm 4.13]{DvZa1}, it was shown that $\mathcal{SM}(N)$ is the one-point compactification of $\mathcal{SM}(M)$ whenever $M$ has no top element. Of course,  this is not true for the state-spaces $\mathcal S(N)$ and $\mathcal S(M)$ because if we take a state-morphism $s$ on $M$, then $s_\lambda =\lambda \tilde s+(1-\lambda) s_\infty$ for each $\lambda \in (0,1)$ is a state on $N$. We have uncountably many different states on $N$ and $s_\lambda(x)=\lambda s(x)$ if $x \in M$ is not a state on $M$; it is only a pre-state.

\section{Jordan Signed Measures}

In the section, we define Jordan signed measures and strong Jordan signed measures on EMV-algebras. We show that they form a Dedekind $\sigma$-complete $\ell$-groups.

We extend the notion of a state and a pre-state as follows. A mapping $m: M \to \mathbb R$ is said to be a {\it signed measure} if $m(x+y)=m(x)+m(y)$. If a signed measure $m$ is positive, then $m$ is said to be a {\it measure}, and if $m$ is a difference of two measures, $m$ is said to be a {\it Jordan signed measure}. Whence a pre-state is a measure with values in the interval $[0,1]$. We denote by $\mathcal J(M)$ the set of Jordan signed measures.

\begin{proposition}\label{pr:3.1}  Let $M$ be an EMV-algebra and let $d:M\to
\mathbb R$ be a subadditive mapping, i.e. $d(x+y) \le d(x)+d(y)$. For all $x\in M$, assume that the set
\begin{equation}\label{eq:(3.1)}
D(x):=\{d(x_1)+\cdots+d(x_n): x = x_1+\cdots+x_n, \ x_1,\ldots,x_n
\in M\} 
\end{equation}
is bounded above in $\mathbb R$.  Then there is a signed measure $m$ on $M$ such that $m(x)=\bigvee D(x):=\sup D(x)$ for all $x\in M$.

\end{proposition}

\begin{proof}
By (\ref{eq:(3.1)}), $m(x):=\bigvee D(x)$ is a well-defined mapping for all $x
\in M$. It is clear that $m(0)=0$ and now we are  going to show
that $m$ is additive on $M$.

Let $x,y \in M$ with $x+y \in M$ be given. For all decompositions
$$ x = x_1 +\cdots+x_n \ \mbox{and} \ y=y_1+\cdots +y_k$$
with all $x_i,y_j \in M$, we have $x+y = x_1+\cdots+x_n +
y_1+\cdots + y_k$, which yields

$$ \sum d(x_i)+\sum d(y_j) \le m(x+y).
$$
Therefore, $u+v \le m(x+y)$ for all $u\in D(x)$ and $v\in D(y)$.
Since $\mathbb R$ is a Dedekind complete Abelian $\ell$-group, $+$ distributes over an existing $\bigvee=\sup$:
\begin{eqnarray*}
m(x)+m(y)&=& \left(\bigvee D(x)\right) +m(y) = \bigvee_{u \in D(x)} (u+m(y))\\
&=& \bigvee_{u\in D(x)} \left(u+ \left(\bigvee D(y)\right)\right) =
\bigvee_{u\in
D(x)}\bigvee_{v \in D(y)} (u+v)\\
&\le& m(x+y).
\end{eqnarray*}

Conversely, let $x+y=z_1+\cdots+z_n$, where each $z_i \in M$. Then
RDP holding also in $M$ implies that there are elements $x_1,\ldots,x_n, y_1,\ldots, y_n
\in M$ such that $x = x_1+\cdots+x_n$, $y = y_1+\cdots+y_n$ and
$z_i = x_i+y_i$ for $i=1,\ldots,n$.  This yields
$$
\sum_i d(z_i) \le \sum_i (d(x_i)+d(y_i)) = \left(\sum_i
d(x_i)\right) + \left(\sum_i d(y_i)\right) \le m(x)+m(y),
$$
and therefore, $m(x+y)\le m(x)+m(y)$ and finally, $m(x+y)=m(x)+(y)$
for all $x,y \in M$.
\end{proof}

We remind that for two Jordan signed measures $m_1$ and $m_2$ on $M$, we write $m_1 \le^+ m_2$ if $m_1(x)\le m_2(x)$ for all $x\in M$. Now we exhibit the lattice properties of the set of Jordan signed measures on $M$ with respect to the partial order $\le^+$.

\begin{theorem}\label{th:3.4} Let $M$ be an EMV-algebra. For the set $\mathcal J(M)$ of Jordan signed measures on $M$ we have:

\begin{enumerate}

\item[(a)] $\mathcal J(M)$ is a Dedekind complete $\ell$-group with respect to $\le^+$.

\item[(b)] If $\{m_i\}_{i\in I}$ is a non-empty set of $\mathcal J(M)$
that is bounded above, and if $d(x)=\bigvee_i m_i(x)$ for all $x \in
M$, then
$$ \left(\bigvee_i m_i\right)(x) = \bigvee\{d(x_1)+\cdots + d(x_n):
x= x_1+\cdots + x_n, \ x_1,\ldots, x_n \in M\}
$$
for all $x \in M$.

\item[(c)] If $\{m_i\}_{i\in I}$ is a non-empty set of $\mathcal J(M)$
that is bounded below, and if $e(x)=\bigwedge_i m_i(x)$ for all $x
\in M$, then
$$ \left(\bigwedge_i m_i\right)(x) = \bigwedge\{e(x_1)+\cdots + e(x_n):
x= x_1+\cdots + x_n, \ x_1,\ldots, x_n \in M\}
$$
for all $x \in M$.

\end{enumerate}
\end{theorem}

\begin{proof} Let $g \in \mathcal J(M)$ be an upper bound for $\{m_i\}$.
For any $x \in M$, we have $m_i(x)\le g(x)$, so that the mapping
$d(x)=\bigvee_i m_i(x)$ defined on $M$ is  a subadditive mapping.
For any $x \in M$ and any decomposition $x = x_1+\cdots + x_n$
with all $x_i \in M$, we conclude $d(x_1)+\cdots+ d(x_n)\le
m(x_1)+\cdots + m(x_n)=g(x)$.  Hence, $g(x)$ is an upper bound for
$D(x)$ defined by (\ref{eq:(3.1)}).

Proposition \ref{pr:3.1} entails there is a signed measure $m:M \to \mathbb R$ such that $m(x)=\bigvee D(x)$. For every $x \in M$ and
every $m_i$ we have $m_i(x)\le d(x)\le m(x)$, which gives $m_i \le^+
m$. The mappings $m-m_i$ are positive measures belonging to
$\mathcal J(M)$, which gives $m \in \mathcal J(M)$. If $h \in \mathcal J(M)$ such that $m_i\le^+ h$ for any $i\in I$, then $d(x)\le h(x)$ for any $x \in
M$. As above, we can show that $h(x)$ is also an upper bound for
$D(x)$, whence $m(x)\le h(x)$ for any $x \in M$, which gives $m\le
^+h$. Alias, we have proved that $m$ is the supremum of
$\{m_i\}_{i\in I}$, and its form is given by (b).

Now if we apply  the order anti-automorphism $z\mapsto - z$ in $\mathbb R$,
we see that if the set $\{m_i\}_{i\in I}$  in $\mathcal J(M)$ is bounded
below, then it has an infimum given by (c).

It is clear that $\mathcal J(M)$ is directed.  Combining (b)
and (c), we see that $\mathcal J(M)$ is a Dedekind complete $\ell$-group.
\end{proof}

For joins and meets of finitely many Jordan signed measures, Theorem
\ref{th:3.4} can be reformulated as follows.

\begin{theorem}\label{th:3.5}
If $M$ is an EMV-algebra, then the group
$\mathcal J(M)$ of all Jordan signed measures on $M$ is an Abelian
Dedekind complete lattice ordered real vector space. Given
$m_1,\ldots, m_n \in \mathcal J(M)$,

\begin{eqnarray*}
\left( \bigvee_{i=1}^n m_i\right)(x) = \sup\{m_1(x_1)+\cdots
+m_n(x_n)\colon
x = x_1+\cdots +x_n,\ x_1,\ldots, x_n \in M\},\\
\left( \bigwedge_{i=1}^n m_i \right)(x) = \inf\{m_1(x_1)+\cdots
+m_n(x_n)\colon
x = x_1+\cdots +x_n,\ x_1,\ldots, x_n \in M\},\\
\end{eqnarray*}
for all $x \in M$.
\end{theorem}

\begin{proof}
Due to Theorem \ref{th:3.4}, $\mathcal J(M)$ is an Abelian Dedekind
complete $\ell$-group. It is evident that it is a Riesz space, i.e.,
a lattice ordered real vector space.

Take $m_1,\ldots,m_n \in \mathcal J(M)$ and let $m = m_1\vee \cdots
\vee m_n$. For any $x \in M$ and $x=x_1+\cdots+x_n$ with
$x_1,\ldots, x_n \in M$, we have  $m_1(x_1)+\cdots + m_n(x_n) \le
m(x_1)+\cdots + m(x_n) = m(x)$.  Due to Theorem \ref{th:3.4}, given
an arbitrary real number $\epsilon >0$, there is a decomposition $x
= y_1+\cdots+y_k$ with $y_1,\ldots,y_k \in M$ such that

$$ \sum_{j=1}^k \max\{m_1(y_j),\ldots,m_n(y_j)\} > m(x)-\epsilon.
$$

If $k < n$, we can add the zero elements to the decomposition, if
necessary, so that without loss of generality, we can assume that
$k\ge n$.

We decompose the set $\{1,\ldots,k\}$ into mutually non-empty disjoint sets
$J(1),\ldots,J(n)$ such that
$$J(i):=\{j \in \{1,\ldots,k\}\colon \max\{m_1(y_j),\ldots, m_n(y_j)\}=
m_i(y_j)\}.
$$
and
Then $J(i)=\{j_{t_1},\ldots,j_{t_{n_i}}\}$. For every $i=1,\ldots,n$, let $x_i= \sum\{y_k\colon k \in J_i\}$, then $x = x_1+\cdots + x_n$,
$$ \sum_{i=1}^n m_i(x_i) = \sum_{i=1}^n \sum_{j \in
J(i)}m_i(y_j)=\sum_{i=1}^k\max\{m_1(y_j),\ldots, m_n(y_j)\} >
m(x)-\epsilon.$$
This implies $m(x)$ equals the given supremum.

The formula for $(m_1\wedge \cdots \wedge m_n)(x)$ can be obtained
applying   the order anti-automorphism $m\mapsto -m$ holding in
$\mathcal J(M)$.
\end{proof}

If we denote by $\mathcal J_b(M)$ the set of bounded Jordan signed measures, then Theorems \ref{th:3.4}--\ref{th:3.5} hold also for $\mathcal J_b(M)$. We note that if $M$ is a $\sigma$-complete EMV-algebra, then every measure $m$ on $M$ is bounded, and there is an element $x_0\in M$ such that $m(x_0)=\sup\{m(x) \colon x \in M\}$. Indeed, let $r=\sup\{m(x) \colon x \in M\}$. Then there is a sequence of elements $\{x_n\}$ of $M$ such that $r = \sup_n m(x_n)$. Then for $x_0 = \bigvee_n x_n$, we have $r \ge m(x_0) \ge m(x_n)$ so that $r=m(x_0)$ and $r$ is a finite number.

A special interest is devoted to measures $m$ with the following property: There is an element $a\in M$ such that $m(a)=\sup\{m(x)\colon x \in M\}$; such measures are said to be {\it strong measures}. Clearly, if $s$ is a state on $M$ and $r\in [0,\infty)$, then $sr$ is a strong measure. Conversely, if $m$ is a not-vanishing strong measure, then $\frac{1}{r}m$, where $r=\sup\{m(x)\colon x \in M\}$, is a state on $M$. Difference of two strong measure is said to be a {\it strong Jordan signed measure}; we denote by $\mathcal J_s(M)$ the set of strong Jordan signed measures on $M$.

\begin{theorem}\label{th:3.6}
Let $M$ be an EMV-algebra. Then $\mathcal J_s(M)$ is an $\ell$-group which is also a Riesz space.
\end{theorem}

\begin{proof}
Let $m_1$ and $m_2$ be strong measures on $M$. Since $m_i\le^+ m_1+m_2$, $i=1,2$, according to Theorem \ref{th:3.5}, there exists a measure $m=m_1\vee m_2$. We show that $m$ is a strong measure. Being $m_1$ and $m_2$ strong measures, there is an element $a_i \in M$ such that $m_i(a_i)=\sup\{m_i(x)\colon x \in M\}$ for $i=1,2$. Since $M$ is an EMV-algebra, we can assume that there is an idempotent $a\in M$ with $a\ge a_1,a_2$ such that $m_i(a)=m_i(a_i)$ for $i=1,2$.  Let $r = \sup\{m(x)\colon x \in M\}$. Given $\epsilon >0$, there is an element $x\in M$ such that $r<m(x)+\epsilon$. For this $x$, there are elements $x_1,x_2 \in M$ with $x=x_1+x_2$ and $m(x)<m_1(x_1)+m_2(x_2)+\epsilon$. Then
\begin{eqnarray*}
r&<& m(x)+\epsilon <m_1(x_1)+m_2(x_2)+2\epsilon = m_1(x_1\wedge a)+m_2(x_2\wedge a)+2\epsilon\\
&=& m((x_1+x_2)\wedge a)+2\epsilon \le m(a) +2\epsilon,
\end{eqnarray*}
where we have used a fact that $(x_1\wedge a)+(x_2\wedge a)=x\wedge a$ for $a\in \mathcal I(M)$.
Then $r\le m(a)$, so that $r=m(a)$, and $m$ is a strong measure.

Since $m_1,m_2\le^+ m$, we have $m-m_1,m-m_2$ are strong measures on $M$, so that $m_0=(m-m_1)\vee (m-m_2)$ is a strong measure, so that $m-(m_1\wedge m_2)=(m-m_1)\vee (m-m_2)$ which yields $m_1\wedge m_2$ is also a strong measure.

Now let $m_1,m_2 \in \mathcal J_s(M)$. Then $m_i=m_i^+-m_i^-$, so that $m_i\le^+ m_1^+ +m_1^-+m_2^+ + m_2^-$, giving $\mathcal J_s(M)$ is directed with the positive cone consisting of all strong measures. Therefore, $\mathcal J_s(M)$ is an $\ell$-group which is also a Riesz space.
\end{proof}

\section{Integral Representation of States}

In the section, we show that every state on an EMV-algebra, which is a finitely additive function, can be represented by a unique regular Borel probability measure which is a $\sigma$-additive measure on the Borel $\sigma$-algebra on some locally compact Hausdorff space. First such a result for MV-algebras was established in \cite{Kro,Pan} and for effect algebras in \cite{241,247}. In addition, we extend this result also for pre-states.

First we present some notions about simplices. For more info about them see the books \cite{Alf,Goo}.

We recall that a {\it convex cone} in a real linear space $V$ is any
subset $C$ of  $V$ such that (i) $0\in C$, (ii) if $x_1,x_2 \in C$,
then $\alpha_1x_1 +\alpha_2 x_2 \in C$ for any $\alpha_1,\alpha_2
\in \mathbb R^+$.  A {\it strict cone} is any convex cone $C$ such
that $C\cap -C =\{0\}$, where $-C=\{-x:\ x \in C\}$. A {\it base}
for a convex cone $C$ is any convex subset $K$ of $C$ such that
every non-zero element $y \in C$ may be uniquely expressed in the
form $y = \alpha x$ for some $\alpha \in \mathbb R^+$ and some $x
\in K$.

We recall that in view of \cite[Prop 10.2]{Goo}, if $K$ is a
non-void convex subset of $V$, and if we set
$$ C =\{\alpha x:\ \alpha \in \mathbb R^+,\ x \in K\},
$$
then $C$ is a convex cone in $V$, and $K$ is a base for $C$ iff
there is a linear functional $f$ on $V$ such that $f(K) = 1$ iff $K$
is contained in a hyperplane in $V$ which misses the origin.

Any strict cone $C$ of $V$ defines a partial order $\le_C$ on $V$ via $x
\le_C y$ iff $y-x \in C$. It is clear that $C=\{x \in V:\ 0 \le_C
x\}$. A {\it lattice cone} is any strict convex cone $C$ in $V$ such
that $C$ is a lattice under $\le_C$.

A {\it simplex} in a linear space $V$ is any convex subset $K$ of
$V$ that is affinely isomorphic to a base for a lattice cone in some
real linear space. A  simplex $K$ in a locally convex Hausdorff
space is said to be (i) {\it Choquet} if $K$ is compact, and (ii)
{\it Bauer} if $K$ and $\partial K$ are compact, where $\partial
K$ is the set of extreme points of $K$.

\begin{theorem}\label{th:sim1}
Let $M$ be an EMV-algebra. Then the state space $\mathcal S(M)$ is a simplex. In addition, the following equivalences hold:
\begin{itemize}
\item[{\rm(i)}] $M$ has a top element.
\item[{\rm(ii)}] $\mathcal S(M)$ is a Choquet simplex.
\item[{\rm(iii)}] $\mathcal S(M)$ is a Bauer simplex.
\end{itemize}
\end{theorem}

\begin{proof}
By Theorem \ref{th:3.6}, the space $\mathcal J_s(M)$ of strong Jordan signed measures on $M$ is an $\ell$-group. The space of strong Jordan measures is its positive cone whose base is the set of states on $M$. Therefore, $\mathcal S(M)$ is a simplex.

Since $\partial \mathcal S(M)=\mathcal{SM}(M)$, see Theorem \ref{th:state8}, due to Proposition \ref{pr:state9}, we have that (i)--(iii) are all mutually equivalent statements.
\end{proof}

Nevertheless the state space $\mathcal S(M)$ is not always a Bauer
simplex,  $\partial {\mathcal S}(M)=\mathcal{SM}(M)$ is always a
Baire space, i.e. the intersection of any sequence of open dense subsets is dense. This was established for EMV-algebras in \cite[Cor 4.12]{DvZa1}.

Let ${\mathcal B}(K)$ be the Borel $\sigma$-algebra of a Hausdorff topological space $K$ generated by all open subsets of $K$. Every element of ${\mathcal B}(K)$ is said to be a {\it Borel set} and each $\sigma$-additive (signed) measure on it is said to be a {\it Borel (signed) measure}. We recall that a Borel measure $\mu$ on $\mathcal B(K)$ is called {\it regular} if
\begin{equation}\label{eq:regular}
\inf\{\mu(O):\ Y \subseteq O,\ O\ \mbox{open}\}=\mu(Y)
=\sup\{\mu(C):\ C \subseteq Y,\ C\ \mbox{compact}\}
\end{equation}
for any $Y \in {\mathcal B}(K)$. For example, let $\delta_x$ be the Dirac measure concentrated at the point $x \in K$, i.e., $\delta_x(Y)= 1$ iff $x \in Y$, otherwise $\delta_x(Y)=0$, then every Dirac measure is a regular Borel probability measure whenever $K$ is compact, see e.g. \cite[Prop 5.24]{Goo}.

Let $K$ be a locally compact Hausdorff topological space. Due to the Alexander theorem, see \cite[Thm 4.21]{Kel}, there is the one-point compactification of $K$, which is a space $K\cup\{x_\infty\}$, where $x_\infty \notin K$. In  \cite[Thm 4.13]{DvZa1}, it was shown that if $M$ has no top element, the one-point compactification of $\mathcal{SM}(M)$ is homeomorphic to the set $\mathcal{SM}(N)$.

\begin{theorem}\label{th:integ}{\rm [Integral Representation of States]}
Let $M$ be an EMV-algebra and let $s$ be a state on $M$. Then there is a unique regular Borel probability measure $\mu_s$ on the Borel $\sigma$-algebra $\mathcal B(\mathcal S(M))$ such that
\begin{equation}\label{eq:integ}
s(x)=\int_{\mathcal{SM}(M)} \hat x(t) \dx \mu_s(t),\quad x\in M,
\end{equation}
where $\hat x$ $(x\in M)$ is a continuous affine mapping from $\mathcal S(M)$ into the interval $[0,1]$ such that $\hat x(s):=s(x)$, $s \in \mathcal S(M)$.

Morevover, there is a one-to-one correspondence between the set of regular Borel probability measures on $\mathcal B(\mathcal S(M))$, and the set of regular Borel probability measures on $\mathcal B(\mathcal S(N))$ vanishing at $\{s_\infty\}$.
\end{theorem}

\begin{proof}
Let $M\ne \{0\}$ be an EMV-algebra and let $N$ be its representing MV-algebra. Then $\mathcal S(M)$ is a non-empty convex set closed in the weak topology of states which is not compact whenever $M$ has no top element. If $M$ has a top element, the statement follows from \cite{Kro,Pan}.

In what follows, we show that the statement is valid also in the case that $M$ has no top element.

Given $x \in M$, we define a mapping $\hat x: \mathcal S(M) \to [0,1]$ by $\hat x(s):=s(x)$, $s \in \mathcal S(M)$. Then $\hat x$ is  continuous and affine, i.e. it preserves convex combinations. Every $\hat x$ can be uniquely extended by Proposition \ref{pr:state11} to a continuous affine mapping $\bar{\hat x}$ defined on the compact convex set $\mathcal S(N)$.

In a similar way, for each $y \in N$, we define $\hat y:\mathcal S(N)\to [0,1]$ such that $\hat y(s)=s(y)$, $s \in \mathcal S(N)$. If $y \in M$, then $\bar{\hat y}(\tilde s) = \hat y(s)$, $s \in \mathcal{SM}(M)$.

Since $N$ is an MV-algebra and $\mathcal S(N)$ is a convex compact set, by \cite{Kro,Pan}, see also \cite{241,247}, for every state $s\in \mathcal S(N)$, there is a unique regular Borel probability measure $\mu_s$ such that
$$
s(y)=\int_{\mathcal{SM}(N)} \hat y(u)\dx \mu_s(u),\quad y \in N.
$$

Define a mapping $\phi: \mathcal{SM}(M)\to \mathcal{SM}(N)$ by $\phi(s)=\tilde s$, $s\in \mathcal{SM}(M)$, where $\tilde s$ is given by (\ref{eq:state}). Then $\phi$ is injective, affine, continuous, and open. In this case, $X:=\mathcal{SM}(M)$ is a locally compact Hausdorff space and $X^*:=\mathcal{SM}(N)$ is a compact Hausdorff space. Therefore, for the Borel $\sigma$-algebra $\mathcal B(\mathcal{SM}(M))$, we have $\mathcal B(\mathcal{SM}(M)) = \phi^{-1}(\mathcal B(\mathcal{SM}(N)))$.

Now let $s$ be a state on $M$ and let $\tilde s$ be its unique extension to a state on $N$ defined by (\ref{eq:state}), and let $x \in M$. Then there is a unique regular Borel measure $\mu_{\tilde s}$ on $\mathcal B(\mathcal{SM}(N))$ such that we have for all $x \in M$
\begin{eqnarray*}
s(x)&=&\tilde s(x)= \int_{\mathcal{SM}(N)} \bar{\hat x}(u)\dx \mu_{\tilde s}(u)=
\int_{\{s_\infty\}} \bar{\hat x}(u)\dx \mu_{\tilde s}(u) + \int_{\mathcal{SM}(N)\setminus \{s_\infty\}} \bar{\hat x}(u)\dx \mu_{\tilde s}(u)\\
&=& \int_{\phi(\mathcal{SM}(M))} \bar{\hat x}(u)\dx \mu_{\tilde s}(u),
\end{eqnarray*}
when we have used the fact $\bar{\hat x}(s_\infty)=s_\infty(x)=0$.

Now let $a\in M$ be an idempotent such that $s(a)=1$. Define $S(a)=\{s\in \mathcal{SM}(M)\colon s(a)>0\}=\{s\in \mathcal{SM}(M)\colon s(a)=1\}$. Then $S(a)$ is compact and open by \cite[Thm 4.10]{DvZa1}. Similarly, if $S_N(a)=\{s\in \mathcal{SM}(N)\colon s(a)>0\}$, then $S_N(a)$ is also compact and open, $\phi(S(a))=S_N(a)$, and $s_\infty \notin S_N(a)$. Therefore, $\bar{\hat a}=\chi_{S_N(a)}$, so that
$$
1=s(a)=\int_{\mathcal{SM}(N)} \bar{\hat a}(t)\dx \mu_{\tilde s}(t)= \int_{\mathcal{SM}(N)}\chi_{S_N(a)}(t)\dx \mu_{\tilde s}(t) =\mu_{\tilde s}(S_N(a)),
$$
which implies $\mu_{\tilde s}(\{s_\infty\})=0$ for each $s\in \mathcal{SM}(M)$.

Now let $\mu$ be a regular Borel measure on $\mathcal{SM}(N)$ such that $\mu(\{s_\infty\})=0$.  Define a mapping $\mu_\phi: \mathcal B(\mathcal{SM}(M))\to [0,1]$ by $\mu_\phi(A)=\mu(\phi(A))$, $A \in \mathcal B(\mathcal{SM}(M))$. Since $A = \phi^{-1}(A_N)$ for some $A_N \in \mathcal B(\mathcal{SM}(N))$, we have $\phi(A)=A_N\setminus \{s_\infty\}\in \mathcal B(\mathcal{SM}(N))$, which shows that $\mu_\phi$ is a Borel probability measure on $\mathcal B(\mathcal{SM}(M))$.   Now we show that $\mu_\phi$ is a regular measure. Let $Y \in \mathcal B(\mathcal{SM}(M))$. If $O_N$ is an open set in $\mathcal{SM}(N)$ such that $\phi(Y)\subseteq O_N$, then $Y \subseteq \phi^{-1}(O_N)$ and if $Y \subseteq O$, where $O$ is an open set in $\mathcal{SM}(M)$, then $\phi(Y)\subseteq \phi(O)$ and $\phi(O)$ is open in $\mathcal{SM}(N)$.  Therefore,
\begin{eqnarray*}
\mu_\phi(Y)&=&\mu(\phi(Y))=\inf\{\mu(O_N)\colon \phi(Y)\subseteq O_N, O_N \text{ open in } \mathcal{SM}(N)\}\\
&=&\inf\{\mu(O_N\setminus\{s_\infty\})\colon Y \subseteq \phi^{-1}(O_N), O_N \text{ open in } \mathcal{SM}(N)\}\\
&=&\inf\{\mu_\phi(\phi^{-1}(O_N))\colon Y \subseteq \phi^{-1}(O_N),O_N \text{ open in } \mathcal{SM}(N)\}\\
&=&\inf\{\mu_\phi(O)\colon Y \subseteq O, O \text{ open in } \mathcal{SM}(M)\}.
\end{eqnarray*}

Now let $C$ be a compact subset of $\mathcal{SM}(M)$.  Then $X\setminus C$ is open in $X$ and therefore, $s_\infty \notin \phi(X\setminus C)=\phi(X)\setminus \phi(C)$ is open in $X^*$ which means $X^*\setminus (\phi(X)\setminus \phi(C))=\{s_\infty\} \cup \phi(C)$ is closed in $X^*$. In addition, $\phi(C)$ is also closed in $X^*$. Whence, if $C\subseteq Y$, then $\phi(C)\subseteq \phi(Y)$.  Now let $C_N$ be a closed subset of $X^*$ such that $C_N\subseteq \phi(Y)$, then $s_\infty \notin \phi (C_N)\subseteq \phi(Y)$, and $s_\infty \in X^*\setminus C_N$ which means $C:=\phi^{-1}(X^*\setminus (X^*\setminus C_N))=\phi^{-1}(C_N)$ is a compact subset of $X$, and $C\subseteq Y$. Therefore,
\begin{eqnarray*}
\mu_\phi(Y)&=&\mu(\phi(Y))=\sup\{\mu(C_N)\colon  C_N \subseteq \phi(Y), C_N \text{ closed in } \mathcal{SM}(N)\}\\
&=&\sup\{\mu(C_N)\colon  \phi^{-1}(C_N) \subseteq Y, C_N \text{ closed in } \mathcal{SM}(N)\}\\
&=& \sup\{\mu(\phi(C))\colon C \subseteq Y, C \text{ compact in } \mathcal{SM}(M)\}\\
&=& \sup\{\mu_\phi(C)\colon C \subseteq Y, C \text{ compact in } \mathcal{SM}(M)\},
\end{eqnarray*}
which proves that $\mu_\phi$ is a regular Borel probability measure on $\mathcal B(\mathcal{SM}(M))$.

If we put $\mu_s(Y)=\mu_{\tilde s}(\phi(Y))$, $Y \in \mathcal B(\mathcal{SM}(M))$, then $\mu_s$ is a regular Borel probability measure on $\mathcal B(\mathcal{SM}(M))$. Using the transformation of integrals, \cite[p. 163]{Hal}, and the equalities  $\hat x(\phi^{-1}(\tilde t))= \phi^{-1}(\tilde t)(x)=t(x)= \tilde t(x)= \bar{\hat x}(\tilde t)$, where $t \in \mathcal{SM}(M)$, we have
\begin{eqnarray*}
\int_{\mathcal{SM}(M)} \hat x(t) \dx \mu_s(t)&=& \int_{\phi^{-1}(\phi(\mathcal{SM}(M))} \hat x(t)\dx \mu_{\tilde s}(\phi(t))\\
&=& \int_{\phi(\mathcal{SM}(M))} \hat x(\phi^{-1}(\tilde t))\dx \mu_{\tilde s}(\tilde t)\\
&=& \int_{\phi(\mathcal{SM}(M))} \bar{\hat x}(\tilde t)\dx \mu_{\tilde s}(\tilde t)\\
&=&\int_{\phi(\mathcal{SM}(M))} \bar{\hat x}(u)\dx \mu_{\tilde s}(u)=s(x).
\end{eqnarray*}

In what follows, we show that there is a one-to-one correspondence between regular Borel probability measures on $\mathcal B(\mathcal{SM}(M))$ and ones on $\mathcal B(\mathcal{SM}(N))$ vanishing at $\{s_\infty\}$.
Let $\mu$ be an arbitrary regular Borel probability measure on $\mathcal{SM}(M)$. Define a mapping $\mu^\phi(A):=\mu(\phi^{-1}(A))$, $A \in \mathcal B(\mathcal{SM}(N))$. Then $\mu^\phi$ is a Borel probability measure on $\mathcal B(\mathcal{SM}(N))$ such that $\mu^\phi(\{s_\infty\})=0$. For every $A \in \mathcal{SM}(N)$, we have
\begin{eqnarray*}
\mu^\phi(A)= \mu(\phi^{-1}(A))&=& \inf\{\mu(O)\colon \phi^{-1}(A) \subseteq O, O \text{ open in } \mathcal{SM}(M)\}\\
&=& \inf\{\mu(O)(A)\colon A \subseteq \phi(O), O \text{ open in } \mathcal{SM}(M)\}\\
&=& \inf\{\mu(\phi^{-1}(O_N))\colon A \subseteq O_N, O_N \text{ open in } \mathcal{SM}(N)\}\\
&=& \inf\{\mu^\phi(O_N)\colon  A \subseteq O_N, O_N \text{ open in } \mathcal{SM}(N)\},
\end{eqnarray*}
and
\begin{eqnarray*}
\mu^\phi(A)= \mu(\phi^{-1}(A))&=&\sup\{\mu(C)\colon C \subseteq \phi^{-1}(A), C \text{ compact in } \mathcal{SM}(N)\}\\
&=& \sup\{\mu(C)\colon \phi(C) \subseteq A, C \text{ compact in } \mathcal{SM}(N)\}\\
&=&\sup\{\mu^\phi(C_N)\colon C_N \subseteq A, C_N \text{ compact in } \mathcal{SM}(N)\},\\
\end{eqnarray*}
which proves that $\mu^\phi$ is regular. Due to the above, we get $\mu = (\mu^\phi)_\phi$, and $(\nu_\phi)^\phi =\nu$ for every regular Borel probability measure $\nu \in \mathcal B(\mathcal{SM}(N))$ vanishing at $\{s_\infty\}$ which establishes a one-to-one correspondence in question.

Finally, we show the uniqueness of $\mu_s$ in formula (\ref{eq:integ}). Let there be another regular Borel probability measure $\nu$ on  $\mathcal B(\mathcal{SM}(M))$ for which (\ref{eq:integ}) holds, and define $\nu^\phi$ by the above way. Then $\nu^\phi$ is a regular Borel probability measure on
$\mathcal B(\mathcal{SM}(N))$ vanishing at $\{s_\infty\}$. Then in a similar way as it was already used for integral transformation, we obtain for each $x \in M$
\begin{eqnarray*}
\tilde s(x) &=& s(x)= \int_{\mathcal{SM}(M)} \hat x(t) \dx \nu(t)
= \int_{\phi(\mathcal{SM}(M))} \bar{\hat x}(u) \dx \nu^\phi(u)\\
&=& \int_{\mathcal{SM}(N)} \bar{\hat x}(u) \dx \nu^\phi(u)
= \int_{\mathcal{SM}(N)} \bar{\hat x}(u) \dx \mu_{\tilde s}(u).
\end{eqnarray*}
If $x = \lambda_1(x_0)$ for $x_0 \in M$, we have
$$
\tilde s(x)=1-s(x_0) = 1- \int_{\mathcal{SM}(N)} \bar{\hat x}_0(u) \dx \nu^\phi(u) = \int_{\mathcal{SM}(N)} \bar{\hat x}(u) \dx \nu^\phi(u)= \int_{\mathcal{SM}(N)} \bar{\hat x}(u) \dx \mu_{\tilde s}(u),
$$
which yields $\mu_{\tilde s}=\nu^\phi$, i.e. $\mu_s=(\mu_{\tilde s})_\phi=(\nu^\phi)_\phi = \nu$.
\end{proof}

Formula (\ref{eq:integ}) is interesting in the following
point of view: de Finetti in a large number of papers, published as early in the Thirties, ``has always insisted that $\sigma$-additivity is not an integral part of the probability concepts but is rather in the nature of a regularity hypothesis", see \cite[p. vii]{BhBh}. On the other hand, by Kolmogorov \cite{Kol}, a probability measure is assumed to be $\sigma$-additive. The mentioned formula shows that there is a natural coexistence between both approaches.

\begin{remark}\label{rm:integ1}
If an EMV-algebra $M$ has no top element, then for every state-morphisms $s \in \mathcal{SM}(M)$, the unique regular Borel probability measure $\mu_s$ corresponding in formula {\rm(\ref{eq:integ})} to $s$ is the Dirac measure $\delta_s$ concentrated at the point $s$ which is a regular Borel probability measure nevertheless  $\mathcal{SM}(M)$ is only a locally compact and not compact space.
\end{remark}

\begin{proof}
Let $\tilde s$ be the extension of $s$ onto the representing MV-algebra $N$, i.e. $\phi(s)=\tilde s\in \mathcal{SM}(N)$. There exists a unique regular Borel probability measure which is the Dirac measure $\delta_{\tilde s}$ on $\mathcal B(\mathcal{SM}(N))$ corresponding to $\tilde s$ through formula (\ref{eq:integ}). Then from the proof of the latter theorem, we have $\mu_s=\delta_{\tilde s}\circ \phi=\delta_{\phi^{-1}(\tilde s)}=\delta_s$.
\end{proof}

\begin{remark}\label{rm:integ2}
Let $M$ be an EMV-algebra without top element and let $\mu$ be a regular Borel probability measure on $\mathcal B(\mathcal{SM}(M))$. Define a mapping $s_\mu: M \to [0,1]$ given by
\begin{equation}\label{eq:integ1}
s_\mu(x)=\int_{\mathcal{SM}(M)} \hat x(t) \dx \mu(t),\quad x \in M.
\end{equation}
Then $s_\mu$ is a pre-state on $M$. If $M$ has a top element, $s_\mu$ is a state on $M$, if $M$ has no top element, then $s_\mu$ is a state on $M$ if and only if there is an idempotent $a \in M$ such that $\mu(S(a))=1$, where $S(a):=\{s\in \mathcal{SM}(M) \colon s(a)>0\}$.
\end{remark}

\begin{proof}
It is evident that $s_\mu$ is a pre-state on $M$, and if $M$ has a top element, $s_\mu$ is a state. Now let $M$ have no top element. Then $s_\mu$ is a state iff there is an idempotent $a$ such that $s_\mu(a)=1$. Given an idempotent $a\in M$, the set $S(a)$ is both open and compact, see \cite[Thm 4.10]{DvZa1}. Then
$$
s_\mu(a) =\int_{\mathcal{SM}(M)} \hat a(t)\dx \mu(t)= \int_{\mathcal{SM}(M)} \chi_{S(a)}(t)\dx \mu(t)=\mu(S(a)),
$$
which establishes the result.
\end{proof}

A converse to the latter remark will be given in Theorem \ref{th:integ5}. Now we exhibit Example \ref{ex:state9} with respect to the latter remark.

\begin{example}\label{ex:integ3}
Let $\mathcal T$ be the system of all finite subsets of $\mathbb N$. Then there are infinitely many regular Borel probability measures $\mu$ on $\mathcal B(\mathcal{SM}(\mathcal T))$ such that formula {\rm (\ref{eq:integ1})} defines a pre-state that is not strong. In addition, for every state $s$ on $\mathcal T$, there is the least finite subset $A$ of $\mathbb N$ such that $s(A)=1$, and $s$ is of the form $s=\sum\{\lambda_n s_n\colon n \in A\}$ and $\sum\{\lambda_n \colon n \in A\}=1$. Therefore, $\mathcal S(\mathcal T)=\Con(\mathcal{SM}(\mathcal T))$.
\end{example}

\begin{proof}
By Example \ref{ex:state9}, $\mathcal{SM}(\mathcal T)=\{s_n\colon n \in \mathbb N\}$, where $s_n(A)=\chi_A(n)$, $A\in \mathcal T$. Since every $A\in \mathcal T$ is an idempotent, $S(A)=\{s\in \mathcal{SM}(\mathcal T)\colon s(A)>0\}=\{s_n\colon n \in A\}$, which means that for every finite subset $A$ of $\mathbb N$, the set $S(A)$ is compact and open, so that $\mathcal{SM}(\mathcal T)$ is homeomorphic to $\mathbb N$ with the discrete topology. Therefore, $\mathcal B(\mathbb N)=2^{\mathbb N}$. Let $\mu$ be any Borel probability measure on $\mathcal B(\mathbb N)$. Then $1=\mu(\mathbb N)=\sum_n \mu(\{n\})$. If we set $\lambda_n = \mu(\{n\})$ for each $n\ge 1$, then $\mu = \sum_n \lambda_n \mu_n$, where $\mu_n(Y)=\chi_Y(n)$, $Y \in \mathcal B(\mathbb N)$ and $n\ge 1$. Since every set $Y\subseteq \mathbb N$ is open, we have $\mu(Y)=\inf\{\mu(O)\colon Y\subseteq O, O \text{ open in }\mathbb N\}$. On the other hand, every compact set $C$ of $2^{\mathbb N}$ is only a finite subset of $\mathbb N$, we show easily that $\mu(Y) = \sup\{\mu(C)\colon Y \subseteq C, C \text { compact in } \mathbb N\}$, so that every Borel probability measure on $\mathcal B(\mathbb N)$ is regular.

If we take a sequence $\{\lambda_n\}$ of numbers from the real interval $(0,1)$ such that $\sum_{n=1}^\infty \lambda_n = 1$, we have $\mu(S(A))<1$ for each $A \in \mathcal T$, so that $s_\mu$ defined by (\ref{eq:integ1}) is a pre-state, and since $\sup\{s_\mu(A)\colon A \in \mathcal T\}=1$, we see that $s_\mu$ is not strong. In addition, every pre-state on $\mathcal T$ which is not strong is of the form $s_\mu$ for the just described measure $\mu$, and there is infinitely many of such regular Borel probability measures $\mu$.

The characterization of states on $\mathcal T$ follows now from Remark \ref{rm:integ2}.
\end{proof}

Now we characterize states in Example \ref{ex:contra} defined by (\ref{eq:integ1}).

\begin{example}\label{ex:integ4}
Let $\mathcal N=\{A \subseteq \mathbb N \colon \text{ either } A \text{ is finite or } \mathbb N\setminus A \text{ is finite}\}$. Then $\mathcal N$ is an MV-algebra representing $\mathcal T$ from the previous example and $\mathcal{SM}(\mathcal N)=\{\tilde s_n \colon n \in \mathbb N\}\cup\{s_\infty\}$, where $\tilde s_n(A)=\chi_A(n)$ for $A\in \mathcal N$.
Every state $s$ on $\mathcal T$ can be uniquely expressed in the form $s = \sum_n \lambda_n \tilde s_n +\lambda_\infty s_\infty$, where each $\lambda$ is from the interval $[0,1]$ such that the sum of all $\lambda$'s is $1$.

In particular, if $s$ is a state on $\mathcal T$, then $s$ can be uniquely expressed in the form $s =\sum_n \lambda_n s_n$, where $\sum_n\lambda_n =1$, $\lambda_n \in [0,1]$ for each $n\ge 1$, and only finitely many of $\lambda_n$' s is non-zero. If $s$ is a non-zero strong pre-state on $\mathcal T$, then there is a sequence of real numbers $\{\lambda_n\}$ from the interval $[0,1]$ which is zero for all but finite numbers of $n$ with $\sum_n \lambda_n \le 1$ such that $s=\sum_n \lambda_n s_n$; the representation of $s$ via non-zero $\lambda$'s is unique.

A mapping $s$ on $\mathcal T$ is a pre-state that is not strong if and only if $s=\sum_n \lambda_n s_n$, $\lambda_n \in [0,1]$, $\sum_n \lambda_n \le 1$, and infinitely many of $\lambda_n$'s are non-zero.
\end{example}

\begin{proof}
The state-morphism space $\mathcal{SM}(\mathcal N)$ is homeomorphic to the one-point compactification $\mathbb N^*=\mathbb N \cup \{n_\infty\}$ $(n_\infty\notin \mathbb N)$ of $\mathbb N$, where $\mathbb N$ is endowed with the discrete topology. Since $A\subseteq \mathbb N^*$ is open iff either $A\subseteq \mathbb N$ or $n_\infty \in  A$ and $\mathbb N \setminus A$ is finite, the one-point compactification $\mathbb N^*$ of $\mathbb N$ coincides with the discrete topology on $\mathbb N^*$, so that $\mathcal B(\mathbb N^*)=2^{\mathbb N^*}$. Inasmuch as it is compact, every Borel probability measure on $\mathcal B(\mathbb N^*)$ is regular. Moreover, if $\mu$ is a Borel probability measure, then $1=\mu(\mathbb N^*)= \sum_n\mu(\{n\})+ \mu(\{n_\infty\})= \sum_n \lambda_n \mu_n +\lambda_\infty \mu_\infty$, where $\lambda_n =\mu(\{n\})$, $\lambda_\infty =\mu(\{n_\infty\})$, $\mu_n(Y)=\chi_n(Y)$, and $\mu_\infty(Y)= \chi_Y(n_\infty)$ for each $Y \in \mathcal B(\mathbb N^*)$ and for $n\ge 1$.

Clearly, the function $s_\mu$ defined by $s_\mu(x):=\int_{\mathcal{SM}(\mathcal N)} \hat x(t)\dx \mu(t)$ for $x \in \mathcal N$, is a state on $\mathcal N$. Since the function $\hat x$ is defined on a countable set, $\hat x$ can be expressed in the form
$\hat x(t)=\sum_i x_i\chi_{\{i\}}(t)$, $t \in \mathcal{SM}(\mathcal N)$. Whence $s_\mu(x)=\sum_i f_i\mu(\{i\})=\sum_i f_i\sum_n\lambda_n \mu_n(\{i\})= \sum_n\lambda_n \sum_i f_i\mu_n(\{i\})= \sum_n\lambda_n s_{\mu_n}(x)$.

Now let $s$ be a state on $\mathcal T$. Due to Theorem \ref{th:integ}, there is a unique regular Borel measure on $\mathcal B(\mathcal{SM}(\mathcal T))$ such that (\ref{eq:integ1}) holds for each $x \in \mathcal T$. As we have have seen $\mu=\sum_n \lambda_n \mu_n +\lambda_\infty \mu_\infty$ giving $s=\sum_n \lambda_n \tilde s +\lambda_\infty s_\infty$. Let $s= \sum_n \lambda'_n \tilde s +\lambda'_\infty s_\infty$ and some $\lambda_n \ne \lambda'_n$. Then for the Borel measure $\mu' = \sum_n \lambda'_n \mu_n +\lambda'_\infty \mu_\infty$, we get $\mu(\{n\})=\lambda_n \ne \lambda'_n = \mu'(\{n\})$ which yields $\mu \ne \mu'$. This is absurd because $s(x)=\int _{\mathcal{SM}(\mathcal T)} \hat x(t)\dx \mu(t) = \int _{\mathcal{SM}(\mathcal T)} \hat x(t)\dx \mu'(t)$. The same holds if $\lambda_\infty \ne \lambda'_\infty$.

The uniqueness of expression of a state $s$ on $\mathcal T$ as a finite convex combination of state-morphisms on $\mathcal T$ follows from Example \ref{ex:integ3}, the extension of $s$ onto $\mathcal N$, and from the uniqueness of its expression in $\mathcal N$.

If $s$ is a non-zero strong pre-state on $\mathcal T$, then there is a real number $r$ such that $rs$ is a state on $\mathcal T$ which implies the representation of $s$ in question.

The characterization of pre-states on $\mathcal T$ follows from Proposition \ref{pr:int0}, Theorem \ref{th:int2}, and a representation of states on $\mathcal N$ in the upper part of the present proof.
\end{proof}

In the following theorem we show that also for pre-states there is their representation by regular Borel measures via (\ref{eq:integ}). We note that a regular Borel measure on a Borel $\sigma$-algebra $\mathcal B(K)$ of a Hausdorff topological space $K$ is a $\sigma$-additive positive valued mapping on $\mathcal B(K)$, that satisfies the regularity condition (\ref{eq:regular}).

\begin{theorem}\label{th:integ5}
Let $M$ be an EMV-algebra without top element. The set of pre-states on $M$ is a Bauer simplex, and for each pre-state $s$ on $M$,
\begin{itemize}
\item[{\rm (i)}] there is a unique regular Borel probability measure $\mu_{\tilde s}$ on $\mathcal B(\mathcal{SM}(N))$ such that
$$ s(x)=\int_{\mathcal{SM}(N)} \bar{\hat x}(t) \dx \mu_{\tilde s}(t), \quad x \in M,
$$
where $\bar{\hat x}:\mathcal{S}(N)\to [0,1]$ is a continuous mapping defined by $\bar{\hat x}(t)= t(x)$, $x \in N$;
\item[{\rm (ii)}] there is a unique regular Borel measure $\mu_s$ on $\mathcal B(\mathcal{SM}(M))$ such that
\begin{equation}\label{eq:prestate}
s(x)=\int_{\mathcal{SM}(M)} \hat x(t) \dx \mu_s(t),\quad x\in M,
\end{equation}
where $\hat x: \mathcal{PS}(M)\to [0,1]$ is a continuous affine mapping defined by $\hat x(s)=s(x)$, $s \in \mathcal{PS}(M)$.
\end{itemize}
\end{theorem}

\begin{proof}
Let $s$ be a pre-state on $M$ and let $\tilde s$ be its extension to a state on $N$ given by (\ref{eq:state0}) and let $\kappa:\mathcal{PS}(M)\to \mathcal S(N)$ be given by $\kappa(s)=\tilde s$, $s \in \mathcal{PS}(M)$. By Proposition \ref{pr:int0}, $\kappa$ is an affine homeomorphism, therefore, \cite[Thm 22]{Kro} implies $\mathcal{PS}(M)$ is a Bauer simplex.

By Proposition \ref{pr:state11}, the restriction of $\kappa$ onto $\mathcal{SM}(M)$ is the function $\phi$ given by $\phi(s)=\tilde s$ for each $s\in \mathcal{SM}(M)$.

Given $x\in N$, we define a function $\bar{\hat x}: \mathcal S(N)\to [0,1]$ given by $\bar{\hat x}(s)=s(x)$, $s \in \mathcal S(N)$. By  \cite{Kro} or Theorem \ref{th:integ}, there is a unique regular Borel probability measure $\mu_{\tilde s}$ on $\mathcal B(\mathcal{SM}(N))$ such that
$$\tilde s(x)= \int_{\mathcal{SM}(N)} \bar{\hat x}(t) \dx \mu_{\tilde s}(t), \quad x \in N,
$$
which proves (i).

(ii) Similarly as in the proof of Theorem \ref{th:integ}, we have for each $x \in M$
$$
s(x)=\tilde s(x)=\int_{\phi(\mathcal{SM}(M))} \bar{\hat x}(u)\dx \mu_{\tilde s}(u).
$$

Define $\mu_s(Y)=\mu_{\tilde s}(\phi(Y))$ for each $Y \in \mathcal B(\mathcal{SM}(M))$. Using the proof of Theorem \ref{th:integ}, $\mu_s$ is a regular Borel measure on $\mathcal B(\mathcal{SM}(M))$ which is not necessary a probability (i.e. $\mu_s$ is only $\sigma$-additive and positive). If $\hat x: \mathcal{PS}(M)\to [0,1]$ is given by $\hat x(t)=t(x)$, $t \in \mathcal{PS}(M)$, using the equalities  $\hat x(\phi^{-1}(\tilde t))= \phi^{-1}(\tilde t)(x)=t(x)= \tilde t(x)= \bar{\hat x}(\tilde t)$, where $t \in \mathcal{SM}(M)$,
we have
\begin{eqnarray*}
\int_{\mathcal{SM}(M)} \hat x(t) \dx \mu_s(t)&=&
\int_{\phi^{-1}(\phi(\mathcal{SM}(M))} \hat x(t)\dx \mu_{\tilde s}(\phi(t))\\
&=& \int_{\phi(\mathcal{SM}(M))} \hat x(\phi^{-1}(\tilde t))\dx \mu_{\tilde s}(\tilde t)\\
&=& \int_{\phi(\mathcal{SM}(M))} \bar{\hat x}(\tilde t)\dx \mu_{\tilde s}(\tilde t)\\
&=&\int_{\phi(\mathcal{SM}(M))} \bar{\hat x}(u)\dx \mu_{\tilde s}(u)=s(x).
\end{eqnarray*}

Let $\mu $ be another regular Borel measure on $\mathcal B(\mathcal{SM}(M))$ which satisfies (\ref{eq:prestate}). We define a mapping $\tilde \mu$ on $\mathcal B(\mathcal{SM}(N))$ as follows
$$
\tilde \mu(A)=\begin{cases}
\mu(\phi^{-1}(A)) & \text{ if } s_\infty \notin A,\\
\mu(\phi^{-1}(A)) + 1-\mu(\mathcal{SM}(M)) & \text{ if }  s_\infty \in A,
\end{cases} \quad A \in \mathcal B(\mathcal{SM}(N)).
$$
Then $\tilde\mu$ is a Borel probability measure and, due to fact that $\mathcal B(\mathcal{SM}(N))=\{B \in \mathcal B(\mathcal{SM}(N)) \colon \text{ either } B=\phi(A)  \text { or } B=\phi(A) \cup\{s_\infty\} \text{ for some } A \in \mathcal B(\mathcal{SM}(M))\}$, we have $\tilde \mu$ is also regular, and $\tilde \mu(\phi(Y))=\mu(\phi^{-1}(\phi(Y)))=\mu(Y)$ for each $Y \in \mathcal B(\mathcal{SM}(M))$. Given $x \in N$, we have
\begin{eqnarray*}
\int_{\mathcal{SM}(N)}\bar{\hat x}(u) \dx\tilde \mu(u)&=& \int_{\{s_\infty\}}\bar{\hat x}(u) \dx\tilde\mu(u) + \int_{\mathcal{SM}(N)\setminus\{s_\infty\}}\bar{\hat x}(u) \dx\tilde\mu(u)\\
&= &(1-\mu(\mathcal{SM}(M)))s_\infty(x)+ \int_{\phi(\mathcal{SM}(M))} \bar{\hat x}(u) \dx\tilde\mu(u).
\end{eqnarray*}
If $x \in M$, then
$$\int_{\mathcal{SM}(N)}\bar{\hat x}(u) \dx\tilde\mu(u)= \int_{\phi(\mathcal{SM}(M))} \bar{\hat x}(u) \dx\tilde\mu(u)$$
and check
\begin{eqnarray*}
s(x)=\tilde s(x)&=& \int_{\mathcal{SM}(M)} \hat x(t)\dx\mu_s(t)= \int_{\phi^{-1}(\phi(\mathcal{SM}(M))} \hat x(t)\dx\tilde \mu(\phi(t))\\
&=& \int_{\phi(\mathcal{SM}(M))} \hat x(\phi^{-1}(u))\dx\tilde \mu(u)
= \int_{\phi(\mathcal{SM}(M))} \bar{\hat x}(u)\dx \tilde \mu(u)\\
&=& \int_{\mathcal{SM}(N)} \bar{\hat x}(u)\dx \tilde \mu(u).
\end{eqnarray*}
If $x =\lambda_1(x_0)$, where $x_0\in M$. Then
\begin{eqnarray*}
\int_{\mathcal{SM}(N)}\bar{\hat x}(u) \dx\tilde\mu(u)&=& (1-\mu(\mathcal{SM}(M)))s_\infty(x)+ \int_{\phi(\mathcal{SM}(M))} \bar{\hat x}(u) \dx\tilde\mu(u)\\
&=& (1-\mu(\mathcal{SM}(M))) + \int_{\phi(\mathcal{SM}(M))}(1- \bar{\hat x}_0(u)) \dx\tilde\mu(u)\\
&=& (1-\mu(\mathcal{SM}(M))) + \mu(\mathcal{SM}(M))- \int_{\phi(\mathcal{SM}(M))} \bar{\hat x}_0(u) \dx\tilde\mu(u)\\
&=& 1-s(x_0)=\tilde s(x).
\end{eqnarray*}
Combining both last cases, we have $\tilde \mu = \mu_{\tilde s}$ which yields $\mu = \mu_s$, and the proof of (ii) is finished.
\end{proof}

Finally, we present an integral representation of pre-states on an EMV-algebra with a top element.

\begin{corollary}\label{co:integ6}
Let $M$ be an EMV-algebra with top element. For each state-morphism $s$ on $M$ there is a unique regular Borel measure $\mu_s$ on $\mathcal B(\mathcal{SM}(M))$ such that
$$
s(x)=\int_{\mathcal{SM}(M)} \hat x(t) \dx \mu_s(t), \quad x \in M,
$$
where $\hat x(t)=t(x)$ for each $t \in \mathcal{PS}(M)$.
\end{corollary}

\begin{proof}
If $s$ is the zero pre-state on $M$, then for the zero measure $\mu_0$ we have the result. If $s$ is a non-zero pre-state, then $s$ is strong, and $s/s(1)$ is a state on $M$, where $1$ is the top element on $M$. Using Theorem \ref{th:integ}, there is a unique regular Borel probability $\mu$ on $\mathcal B(\mathcal{SM}(M))$ such that $s(x)/s(1)=\int_{\mathcal{SM}(M)} \hat x(t)\dx\mu(t)$. If we set $\mu_s=s(1)\mu$, we obtain the result in question.
\end{proof}

\section{The Horn--Tarski Theorem}

A famous result by Horn and Tarski \cite{HoTa} states that every
state on a  Boolean subalgebra $A$ of a Boolean algebra $B$ can be
extended to a state on $B$, of course not in a unique way. E.g. the
two-element Boolean subalgebra, $\{0,1\}$ has a unique state,
0-1-valued, and it can be extended to many distinct  states on $B$, in general. This result was generalized for MV-algebras in \cite[Thm 6]{Kro1} and in \cite[Thm 6.9]{241} for effect algebras. In what follows, we generalize this result also for EMV-algebras.

\begin{theorem}\label{th:HT}{\rm [Horn--Tarski Theorem]}
Let $M_0$ be an EMV-subalgebra of an EMV-algebra $M$. Then every state on $M_0$ can be extended to a state on $M$, and every state-morphism on $M_0$ can be extended to a state-morphism on $M$.
\end{theorem}

\begin{proof}
There are three cases.
(i) $M_0$ and $M$ have the same top element $1$, then $M_0$ and $M$ are in fact MV-algebras, and the result follows from \cite[Thm 6]{Kro1}. (ii) $M$ has top element $1$ and $1\notin M_0$. Let $N_0=\{x \in M \colon \text{ either } x\in M_0 \text{ or } \lambda_1(x)\in M_0\}$. Then $N_0$ is an MV-algebra representing $M_0$, for more details, see \cite[Thm 5.21]{DvZa}. By Proposition \ref{pr:state11}, $s$ can be extended to a state $\tilde s$ on $N_0$ defined by (\ref{eq:state}). Applying \cite[Thm 6]{Kro1}, $\tilde s$ can be extended to a state on $M$ which gives the result.

(iii) $M$  has no top element. Let $M_0$ be an EMV-subalgebra of an EMV-algebra $M$. Let $N$ be the representing MV-algebra for $M$ where each element of $N$ is either from $M$ or $\lambda_1(x)$ is from $M$. Define $N_0=\{x\in N\colon \text{ either } x \in M_0 \text{ or } \lambda_1(x)\in M_0\}$. Then $N_0$ is an MV-algebra representing $M_0$ and $N_0$ is an MV-subalgebra of $N$.  Let $s$ be a state on $M_0$, we define a state $\tilde s$ on $N_0$ by (\ref{eq:state}). By \cite[Thm 6]{Kro1}, $\tilde s$ can be extended to a state $s'$ on $N$. The restriction $s''$ of $s'$ onto $M$ is a state on $M$ because there is an element $a\in M_0$ such that $s(a)=1$.

Now let $s$ be a state-morphism on $M_0$. Also in this situation, we have three cases. (i) $M_0$ and $M$ have the same top element $1$. Let $I_0=\Ker(s)$, then $I_0$ can be extended to a maximal ideal $I$ on $M$. There is a state-morphism $s'$ on $M$ such that $I=\Ker(s')$. The restriction $s''$ of $s'$ onto $M_0$ is a state-morphism such that $\Ker(s)\subseteq \Ker(s'')$ implying $s=s''$. Therefore, $s'$ is a state-morphism on $M$ that is an extension of the state-morphism $s$ on $M_0$.

(ii) $M$ has top element $1$ and $1\notin M_0$. Then $N_0=\{x \in M \colon \text{ either } x\in M_0 \text{ or } \lambda_1(x)\in M_0\}$ is an MV-algebra
representing $M_0$. Hence, by Proposition \ref{pr:state11}, the extension $\tilde s$ of $s$ onto $N_0$ given by (\ref{eq:state}) is a state-morphism. Using case (ii), we see $\tilde s$ can be extended to a state-morphism $s'$ on $M$. The restriction $s''$ of $s'$ onto $M_0$ gives a state-morphism such that $\Ker(s)=\Ker(s'')$, so that $s=s''$ and $s'$ is an extension of $s$ onto $M$.

(iii) $M$  has no top element and let $\tilde s$ be its extension to $\tilde s$ on $N_0$. It is easy to verify that $\tilde s$ is a state-morphism on $N_0$, see also \cite[Prop 4.4]{DvZa1}. Moreover, $\Ker(\tilde s) = \Ker(s)\cup \Ker_1(s)^*$, where $A^*=\{\lambda_1(a) \colon a \in A\}$, and $\Ker(\tilde s)$ is a maximal ideal of $N_0$. Let $I_0$ be the ideal of $N$ generated by $\Ker(\tilde s)$. Then $1 \notin I_0$ and there is a maximal ideal $I$ of $N$ containing $I_0$, so there is a state-morphism $s'$ on $N$ such that $\Ker(s')=I$. The restriction $s''$ of $s'$ onto $M$ is a state-morphism and the restriction $s'''$ of $s'$ onto $M_0$ is a state-morphism on $M$. Then $\Ker(s)\subseteq \Ker(s''')$ and the maximality of $\Ker(s)$ and $\Ker(s''')$ gives $\Ker(s)=\Ker(s''')$, so that by Theorem \ref{th:state6}, $s=s'''$ and $s''$ is an extension of $s$ onto $M$.
\end{proof}

\section{Conclusion}

EMV-algebras are new algebraic structures introduced by authors in \cite{DvZa} which resemble MV-algebras locally but the top element is not assumed. In the paper we have introduced a state as a $[0,1]$-valued additive function on an EMV-algebra $M$ which attains the value $1$ at some element $x\in M$. A special kind of states are state-morphisms which are EMV-homomorphisms on $M$ with value in the interval $[0,1]$. The state space of any EMV-algebra is non-void whenever $M$ has at least one non-zero element, and it is a convex subset whose extremal states are exactly state-morphisms, Theorem \ref{th:state8}. Under the weak topology of states, the state space of an EMV-algebra is a convex Hausdorff space which is compact or locally compact iff $M$ possesses a top element, Proposition \ref{pr:state9} and Theorem \ref{th:local}. Nevertheless the state space is not compact, every state lies in the weak closure of the convex hull of state-morphisms, this is a Krein--Mil'man-type theorem for states, Theorem \ref{th:state12}.

A weaker form of states are pre-states and strong pre-states. We have defined Jordan signed measures and strong Jordan signed measures, and we have showed that they form a Dedekind complete $\ell$-groups, Theorem \ref{th:3.4} and Theorem \ref{th:3.6}. This allows us to show that every state on an EMV-algebra is represented by a unique regular Borel probability measure on the Borel $\sigma$-algebra of the state space, Theorem \ref{th:integ}. Finally, we show a variant of the Horn--Tarski theorem showing that every state on an EMV-subalgebra can be extended to a state on the EMV-algebra, Theorem \ref{th:HT}.

\end{document}